\documentclass{article}

\usepackage{booktabs, multirow, bm}
\usepackage{amsmath, amssymb, amsthm, mathrsfs}
\usepackage{indentfirst}
\usepackage{graphicx, tikz}
\usepackage{geometry}
\usepackage{caption}
\usepackage{subfigure, booktabs}
\usepackage{hyperref}

\usepackage{cleveref}
\crefname{section}{Section}{Sections}
\crefname{figure}{Figure}{Figures}
\crefname{table}{Table}{Tables}
\crefname{equation}{}{}
\crefname{theorem}{Theorem}{Theorems}
\crefname{lemma}{Lemma}{Lemmas}
\crefname{remark}{Remark}{Remarks}
\crefname{problem}{Inverse Problem}{Subproblems}
\newtheorem{theorem}{Theorem}[section]

\newtheorem{definition}{Definition}[section]

\theoremstyle{definition}
\newtheorem{example}{\noindent Example}
\newcommand{\ii}{\mathrm{i}}

\newcommand{\dd}{\mathrm{d}}
\graphicspath{{figure/}}
\newcommand\keywords[1]{\textbf{Keywords}: #1}

\begin{document}
	
	\title{Mathematical and numerical study of an inverse source problem for the biharmonic wave equation}
	
	\author{
		Yan Chang\thanks{School of Mathematics, Harbin Institute of Technology, Harbin, China. {\it $21B312002@stu.hit.edu.cn$}},
		Yukun Guo\thanks{School of Mathematics, Harbin Institute of Technology, Harbin, China. {\it ykguo@hit.edu.cn} (Corresponding author)},
		Tao Yin\thanks{LSEC, Institute of Computational Mathematics and Scientific/Engineering Computing, Academy of Mathematics and Systems Science, Chinese Academy of Sciences, Beijing 100190, China. {\it yintao@lsec.cc.ac.cn}},
		\ and Yue Zhao\thanks{School of Mathematics and Statistics, Central China Normal University, Wuhan, China. {\it  zhaoyueccnu@163.com}}
	}
	\date{}
	\maketitle
	
	\maketitle
	\begin{abstract}
		In this paper, we study the inverse source problem for the biharmonic wave equation. Mathematically, we characterize the radiating sources and non-radiating sources at a fixed wavenumber. We show that a general source can be decomposed into a radiating source and a non-radiating source. The radiating source can be uniquely determined by Dirichlet boundary measurements at a fixed wavenumber. Moreover, we derive a Lipschitz stability estimate for determining the radiating source. On the other hand, the non-radiating source does not produce any scattered fields outside the support of the source function. Numerically, we propose a novel source reconstruction method based on Fourier series expansion by multi-wavenumber boundary measurements. Numerical experiments are presented to verify the accuracy and efficiency of the proposed method.
	\end{abstract}

	\section{Introduction}
	
	This paper is concerned with the scattering problem in an infinitely thin elastic plate for which the wave propagation can be
	modeled by the following two-dimensional biharmonic wave equation
	\begin{equation}\label{eq:biharmonic}
		\Delta^2 u(x, k) - k^4  u(x, k) = S(x),
	\end{equation}
	where $k>0$ is the wavenumber and $S$ is the source function. We assume that $S\in  L_{\rm comp}^2(\mathbb R^2)$ and $\text{supp}S \subset B_R$
	where $B_R:= \{x\in \mathbb R^2: |x|< R\}$ with $R$ being a positive constant. Denote the boundary of $B_R$ by $\Gamma_R$.
	The following radiation conditions are imposed on $u$ and $\Delta u$:
	\begin{equation}\label{src}
		\lim_{r\to\infty}\sqrt{r} (\partial_r u-\mathrm{i}k u)=0,\quad \lim_{r\to\infty}\sqrt{r} (\partial_r (\Delta u)-\mathrm{i}k (\Delta u))=0
	\end{equation}
	uniformly in all directions $\hat{x} = x/|x|$ with $r = |x|$. 
	For the direct problem, since
	\begin{align*}
		(\Delta^2 - k^4)^{-1}  = \frac{1}{2k^2} \Big( (-\Delta-k^2)^{-1} -  (-\Delta+k^2)^{-1} \Big),
	\end{align*}
	the fundamental solution of the biharmonic wave equation takes the form
	\begin{align*}
		G(x, y, k) = \frac{\rm i}{8k^2} \left(H_0^{(1)}(k|x - y|) - H_0^{(1)}({\rm i}k|x - y|)\right).
	\end{align*}
	Here $H_0^{(1)}$ denotes the first kind Hankel function with order zero. Then the solution to the direct scattering problem \eqref{eq:biharmonic}--\eqref{src} is given by
	\begin{align}\label{eq:solution1}
		u(x, k) = \int_{\mathbb R^2}  G(x, y, k) S(y){\rm d}y.
	\end{align}
	In this work, we are interested in the inverse problem of identifying 
	the source function from boundary measurements.
	
	The biharmonic wave equation arises in the theory of plate bending and thin plate elasticity \cite{DL23, GGS}.  Although there has been extensive literature on the inverse scattering problems of acoustic, elastic, and electromagnetic waves, the inverse problems of the biharmonic wave equation are much less studied. Recently, due to their important applications in many scientific areas including offshore runway design, seismic cloaks, and platonic crystal \cite{FGE, LW, MMM, WUW},  the inverse problems 
	of the biharmonic wave equation have received considerable attention \cite{LW-22, LW, LYZ, LYZ1, TS}. However, compared with the well-developed inverse theory for acoustic, elastic, and electromagnetic wave equations,
	the inverse scattering problems for the biharmonic wave equation are much less studied. Motivated by the significant applications,
	we investigate the inverse source problem for the biharmonic wave equation. This work contains two main contributions.
	First, the radiating and non-radiating sources are characterized at a fixed wavenumber mathematically. We further derive a Lipschitz stability estimate for determining the radiating source.  
	Second, a direct and effective numerical method is developed to reconstruct the source function
	by multi-wavenumber boundary measurements. 
	
	For the inverse source problems at a single wavenumber, in general the uniqueness cannot be guaranteed due to the existence of non-radiating sources  \cite{blz}.
	The radiating sources and non-radiating sources were mathematically characterized in \cite{Monk} for Maxwell equations. 
	The results in \cite{Monk} were then extended to acoustic and elastic wave equations in \cite{ACTV, KW}. In this paper, we investigate the inverse source problem of the fourth-order biharmonic wave equation. We prove that the source can be decomposed into a radiating source and a non-radiating source. The radiating source can be uniquely identified by Dirichlet boundary measurements at a fixed wavenumber. We further derive a Lipschitz stability estimate for determining the radiating source. 
	The proof of the stability is unified and can be extended to acoustic, elastic, and electromagnetic wave equations.
	On the other hand, the non-radiating source does not produce any scattered fields outside the support of the source function which thus could not be identified.
	We point it out that the analysis for the biharmonic wave equation is more involved due to the increase of the order of the elliptic operator.

	For numerical methods, 
	recently a direct and effective Fourier-based method has been proposed to reconstruct source functions for acoustic, electromagnetic, and elastic waves using data at discrete multiple wavenumbers. We refer the reader to 
	\cite{WangMaGuoLiJDE, IP15, SongWang} and references therein for relevant studies.
	In this paper, we apply this Fourier method to
	investigate the more sophisticated biharmonic wave equation. 
	We would like to emphasize that the increase of the order brings many challenges to the computation. 
	To deal with this difficulty, we introduce two auxiliary functions based on a splitting of the biharmonic wave operator, which enables us to convert the fourth-order equation to two second-order equations.
	
	The rest of this paper is arranged as follows. \Cref{sec2:radiating} provides a characterization the radiating and non-radiating sources at a fixed wavenumber. In \Cref{sec:Fourier}, we develop a computational method to reconstruct the source function based on the Fourier expansion.  
	Numerical experiments are presented in \Cref{sec:numerical} to verify the effectiveness of the proposed method.

	\section{Characterization of the radiating and non-radiating sources at a fixed wavenumber}
	\label{sec2:radiating}
	
	In this section, we characterize the radiating and non-radiating sources. We start with the radiating sources.
	We say that $v\in H^2(B_R)$ is a weak solution to 
	the following homogeneous equation
	\begin{align}\label{homo}
		\Delta^2 v - k^4 v = 0
	\end{align}
	in the distributional
	sense if for every $\varphi\in C_0^\infty(B_R)$ one has
	\begin{align*}
		\int_{B_R} \Delta v \Delta \varphi - k^4 v \varphi {\rm d}x = 0.\label{eq:weak}
	\end{align*}
	Denote the set of weak solutions to \eqref{homo} by $\mathcal H(B_R)$ and denote by $H(B_R)$ the closure  of $\mathcal H(B_R)$ in the $L^2(B_R)$ norm. Then one has the following $L^2$-orthogonal decomposition 
	\[
	L^2(B_R) = H(B_R) + H^\perp(B_R).
	\]
	
	In the following theorem we characterize all the radiating sources which are contained in the functional space $H(B_R)$.
	\begin{theorem}\label{thm:2.1}
		Let $S \in H(B_R)$.
		The Dirichlet boundary measurements $u\vert_{\Gamma_R}, \Delta u\vert_{\Gamma_R}$ at a fixed wavenumber uniquely determine $S$. 
	\end{theorem}
	
	\begin{proof}
		
		It suffices to prove that $u = \Delta u=0$ on $\Gamma_R$ gives $S\equiv 0$.
		Let $v\in \mathcal{H}(B_R)$. Multiplying both sides of \eqref{eq:biharmonic} by $v$ and integrating by parts twice over $B_R$ yield
		\begin{align*}
			\int_{B_R} S v {\rm d}x = \int_{\Gamma_R} \left(\partial_{\nu}\Delta u v -  \partial_{\nu} v \Delta u +  \partial_{\nu} u \Delta v - \partial_{\nu} \Delta v u \right){\rm d}s(x).
		\end{align*}
		Given $u = \Delta u = 0$ on $\Gamma_R$, we also have $\partial_{\nu} u = \partial_{\nu} \Delta u = 0$ on $\Gamma_R$ by the uniqueness result of 
		the exterior scattering problem in $\mathbb R^2\setminus B_R$ (cf \cite[Theorem A.1]{GLL}). Therefore, we have
		\[
		\int_{B_R} S v {\rm d}x  = 0.
		\]
		Since $\mathcal H(B_R)$ is dense in $H(B_R)$ in the norm of $L^2(B_R)$, using a density argument we have $S \equiv 0$.
		The proof is completed.
	\end{proof}

	We further derive a Lipschitz stability estimate of determining the radiating sources with the boundary measurements 
	\[
	u, \, \Delta u, \, \partial_\nu u, \, \partial_\nu \Delta u\quad \text{on}  \quad \Gamma_R
	\]
	at a fixed wavenumber. This method is unified and thus can be extended to study wave equations of second order.
	
	\begin{theorem}
		Let $S\in H(B_{R})$. It holds that
		\begin{align}
			\|S\|_{L^2(B_R)} \leq C\left(\int_{\Gamma_R}(|u|^2 + |\Delta u|^2 + |\partial_\nu u|^2 + |\partial_\nu \Delta u|^2){\rm d}s(x)\right)^{1/4},
		\end{align}
		where $C$ is a generic positive constant depending on $k.$
	\end{theorem}
	
	\begin{proof}
		
		Since $S\in H(B_{R})$, there exists a sequence $\{S_j\}_{j=1}^\infty$ of weak solutions to \eqref{homo} in $H^2(B_{R})$ such that
		\[
		\lim_{j\to\infty} S_j = S \quad \text{in} \,\, L^2(B_{R}).
		\]
		Assume that $u_j$ is the radiating solution to \eqref{eq:biharmonic} corresponding to the source $S_j$, i.e.,
		\begin{equation}\label{eqn_1}
			\Delta^2 u_j - k^4 u_j = S_j.
		\end{equation}
		Since 
		\begin{align}\label{eqn_2}
			\Delta^2 S_j - k^4 S_j= 0 \quad \text{in} \quad B_{R}
		\end{align}
		and $\|S_j\|_{L^2(B_{R})}\leq C$,
		we have $\|\Delta^2S_j\|_{L^2(B_{R})}\leq C$ which gives
		$\|S_j\|_{H^4(B_R)}\leq C$. Here $C$ is a generic constant depending on $k$.

		Multiplying both sides of \eqref{eqn_1} by $\bar{S}_j$ and integrating by parts over $B_R$ yields
		\begin{align*}
			\int_{B_R} |S_j|^2{\rm d}x &= \int_{\Gamma_R} (\bar{S}_j \partial_\nu \Delta u_j  -  \Delta u_j\partial_\nu \bar{S}_j) {\rm d} s(x)\\
			&\quad + \int_{\Gamma_R} (\Delta\bar{S}_j \partial_\nu  u_j  -   u_j\partial_\nu \Delta\bar{S}_j) {\rm d} s(x).
		\end{align*}
		Since $\|S_j\|_{H^4(B_R)}\leq C$,
		by Schwartz's inequality we get
		\begin{align}\label{ineq}
			\int_{B_R} |S_j|^2{\rm d}x \leq C\left(\int_{\Gamma_R} (|u_j|^2 + |\Delta u_j|^2 + |\partial_\nu u_j|^2 + |\partial_\nu \Delta u_j|^2) {\rm d}s(x)\right)^{1/2}.
		\end{align}
		On the other hand, subtracting \eqref{eq:biharmonic} by \eqref{eqn_1} gives
		\[
		\Delta^2 (u - u_j) - k^4 (u - u_j) = S - S_j.
		\]
		Then the resolvent estimate \cite[Theorem~2.1]{LYZ} implies that
		\[
		\|u - u_j\|_{H^4(B_R)} \leq C\|S- S_j\|_{L^2(B_{R})}.
		\]
		Hence, $u_j\to u$ in $H^4(B_R)$ which gives by letting $j\to\infty$ in \eqref{ineq}
		that
		\[
		\|S\|^2_{L^2(B_R)} \leq C\left(\int_{\Gamma_R}\left( |u|^2 + |\Delta u|^2 + |\partial_\nu u|^2 + |\partial_\nu \Delta u|^2\right) {\rm d}s(x)\right)^{1/2}.
		\]
		The proof is completed.
	\end{proof}
	
	
	In the following theorem we prove that the orthogonal complement functional space $H^\perp(B_R)$ of $H(B_R)$ consists of all the non-radiating sources.
	
	\begin{theorem}\label{nrs}
		Let $S\in H^\perp(B_R)$ in the scattering problem \eqref{eq:biharmonic}--\eqref{src}. It holds that $u = \Delta u = \partial_\nu u = \partial_\nu\Delta u = 0$ on $\Gamma_R$.  
	\end{theorem}
	
	\begin{proof}
		
		Let $w\in H^2_{\rm loc}(\mathbb R^2\setminus \Gamma_R)$ be a solution to the following transmission problem
		\begin{align*}
			\begin{cases}
				\Delta^2 w - k^2 w = 0 \quad \text{in} \, \,  \mathbb R^2\setminus \Gamma_R,\\
				[w]_{\Gamma_R} = \xi_1, \quad [\Delta w]_{\Gamma_R} = \xi_2, \quad \text{on} \,\, \Gamma_R,\\
				[\partial_\nu w]_{\Gamma_R} = \eta_1, \quad [\partial_\nu \Delta w]_{\Gamma_R} = \eta_2 \quad \text{on} \,\, \Gamma_R,\\
				w  \, \, \text{and} \, \, \Delta w \, \, \text{satisfy} \, \, \eqref{src}.
			\end{cases}
		\end{align*}
		where $\xi_1\in H^{3/2}(\Gamma_R), \, \xi_2 \in H^{-1/2}(\Gamma_R)$ and $\eta_1\in H^{1/2}(\Gamma_R), \, \eta_2\in H^{-3/2}(\Gamma_R)$ and $[\cdot]_{\Gamma_R}$ denotes the jump of the traces from both sides of $\Gamma_R$. Such a solution can be found through integral representations using only the boundary data 
		$\xi_1,\xi_2,\eta_1,\eta_2$, which is
		analogous to the integral representations in \cite{CK} of the solutions to the classical acoustic
		transmission problems.
		Multiplying both sides of \eqref{eq:biharmonic} by $w$
		and integrating by parts over $B_{r}$ with $r>R$ and then letting $r\to\infty$, we arrive at
		\begin{align*}
			\int_{B_R} S w {\rm d}x = \int_{\Gamma_R} \left(\partial_{\nu}\Delta u \xi_1 -  \eta_1 \Delta u +  \partial_{\nu} u \xi_2 - \eta_2 u\right)\;{\rm d}s(x).
		\end{align*}
		Here we have used the jump conditions on $\Gamma_R$ and the radiating conditions \eqref{src} to eliminate the boundary integral on $\Gamma_r$ by letting $r\to\infty$. 
		Then by the arbitrariness of $\xi_i$ and $\eta_i$ for $i = 1, 2,$ we have
		$u = \Delta u = \partial_\nu u = \partial_\nu\Delta u = 0$ on $\Gamma_R$. The proof is completed.
	\end{proof}
	
	From Theorem \ref{nrs} we see that the non-radiating source does not produce surface measurements on $\Gamma_R$ which thus
	cannot be identified at a fixed wavenumber.

	\section{Fourier-based method to reconstruct the source function}\label{sec:Fourier}
	
	In this section, we develop a Fourier-based method to reconstruct the source function by Dirichlet boundary measurements at multiple wavenumbers.
	As discussed in Introduction, the use of multiple wavenumbers is necessary in order to achieve an accurate reconstruction.

	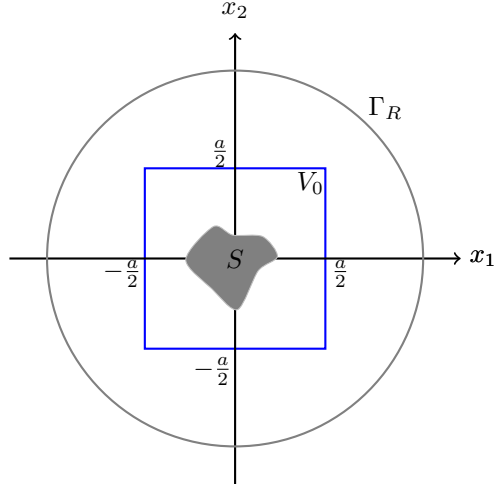
\begin{figure}[htp]
		\centering
		\thispagestyle{empty}
		
		\begin{center}
			\begin{tikzpicture}
				\draw [blue, thick](-1.2,-1.2) rectangle (1.2,1.2);
				\draw [thick,->] (-3,0) -- (3,0);
				\draw [thick,->] (0,-3) -- (0,3) node at (0,3.3) {$x_2$} node at (2,2) {$\Gamma_R$};
				\draw  node at (3.3,0) {$x_1$}  node at (1.4,-0.2) {$\frac{a}{2}$}  node at (-1.5,-0.2) {$-\frac{a}{2}$};
				\draw  node at (3.3,0) {$x_1$}  node at (-.2,1.4) {$\frac{a}{2}$}  node at (-.3,-1.5) {$-\frac{a}{2}$};
				\draw [thick,gray] (0,0) circle (2.5);\draw node at (1,1) {$V_0$};
				\pgfmathsetseed{8}			
				\draw (2,2) plot[gray,smooth cycle, samples=8, domain={1:8}] (2+\x*360/8+3*rnd:.2cm+.5cm*rnd) node at (0, 0) {$S$}[lightgray,fill=gray];
			\end{tikzpicture}
		\end{center}
		\caption{Illustration of the inverse source problem.}\label{fig: setup}
	\end{figure}
	
	\subsection{Fourier approximation}\label{Fourier}
	
	Let ${\bm l}=(l_1,l_2)\in\mathbb{Z}^2$. Denote $\mathbb{K}=\{k_{\bm l}: k_{\bm l} = \frac{2\pi}{a}|\bm l|, \, \bm l\in\mathbb Z^2, \, | \bm l|_\infty\leq N, \, N\in\mathbb N^+\}$.
	In this section,
	we propose a Fourier method to approximate the source function $S(x)$ from the following Dirichlet boundary data at multiple frequencies
	\[
	\{u(x, k),\Delta u(x, k):x\in\Gamma_R,k\in\mathbb{K}\}.
	\]
	For the geometry setup of the inverse source problem, we refer to \Cref{fig: setup} in which the square $V_0$ is set to be
	$$
	V_0=\left(-\frac{a}{2},\frac{a}{2}\right)\times\left(-\frac{a}{2},\frac{a}{2}\right), \ \ a>0,
	$$
	such that $\text{supp}S\subset V_0\subset B_R.$
	
	Denote by $ \phi_{\bm l}(x)$ the Fourier basis functions 
	\begin{align*}
		\phi_{\bm l}(x)=\mathrm{e}^{\mathrm{i}\frac{2\pi}{a}{\bm l}\cdot x},\quad{\bm l}\in\mathbb{Z}^2.
	\end{align*}
	Then the source function $S(x)$ has the following expansion
	\begin{align}\label{FE}
		S(x) = \sum_{\bm l\in\mathbb{Z}^2} \hat{s}_{\bm{l}} \phi_{\bm l}(x).
	\end{align}
	Based on the biharmonic wave equation \eqref{eq:biharmonic},
	we calculate the Fourier coefficients $\hat{s}_{\bm{l}}$ as follows.
	Multiplying both sides of \eqref{eq:biharmonic} by $\overline{ \phi_{\bm l}(x)}$ and integrating by parts over the domain $V_0$ gives 
	\begin{align}\nonumber
		\hat{s}_{\bm{l}}=&\frac{1}{a^2}\int_{V_0}S(x)\overline{ \phi_{\bm{l}}(x)}\mathrm{d}x=\frac{1}{a^2}\int_{B_R}\left(\Delta^2u(x, k_{\bm l})-k^4u(x, k_{\bm l})\right)\overline{ \phi_{\bm{l}}(x)}\mathrm{d}x\\
		\nonumber=&\frac{1}{a^2}\int_{\Gamma_R}\left(\partial_\nu\Delta u(x, k_{\bm l})+\ii\frac{2\pi}{a}({\bm l}\cdot\nu)\Delta u(x, k_{\bm l})\right)\overline{ \phi_{\bm{l}}(x)}\mathrm{d}s(x)\\
		\label{eq:s_l}&-\frac{k^2}{a^2}\int_{\Gamma_R}\left(\partial_\nu u(x, k_{\bm l})+\ii\frac{2\pi}{a}({\bm l}\cdot\nu) u(x, k_{\bm l})\right)\overline{ \phi_{\bm{l}}(x)}\mathrm{d}s(x),\ \ {\bm l}\in\mathbb{Z}^2.
	\end{align}
	The main idea of our method is to approximate the source functions $S$ in \eqref{eq:biharmonic} by the following finite Fourier expansion 
	\begin{align}\label{eq:SN}
		S_N(x)=\sum_{\bm l\in\mathbb K}\hat{s}_{\bm{l}} \phi_{\bm l}(x).
	\end{align}
	
	
	From \eqref{eq:s_l} we see that the Fourier coefficient $\hat{s}_{\bm 0}$ is related to the boundary data corresponding to $k = 0$.
	However, in practice the data at zero wavenumber is hard to achieve. To resolve this issue, we introduce the following admissible wavenumbers.
	The idea is to approximate the Fourier basis $\phi_{\bm 0}$ for $\bm l = 0$ by $\phi_{\bm l_0}$ where the vector $\bm l_0$ is close to $\bm 0$.
	\begin{definition}[Admissible wavenumbers]
		Let $\lambda$ be a small positive constant such that $\frac{2\pi}{a}\lambda<\frac{1}{2}$ and let
		\begin{align*}
			\bm{l}_0:=(\lambda,0).
		\end{align*}
		The admissible wavenumbers are defined by
		\begin{align}\label{eq:wavenumber}
			k_{\bm{l}}:=\left\{
			\begin{array}{cc}
				\frac{2\pi}{a}|\bm{l}|, & {\bm{l}}\in\mathbb{Z}^2\backslash\{\bm{0}\}, \\[3mm]
				\frac{2\pi}{a}\lambda,& {\bm{l}}=\bm l_0.
			\end{array}
			\right.
		\end{align}
	\end{definition}
	
	We approximate the source function $S$ by the Fourier series
	\begin{align*}
		\hat{s}_{\bm l_{0}}\phi_{\bm l_0}+\sum_{|{\bm{l}}|_\infty\ge 1}\hat{s}_{\bm{l}} \phi_{\bm{l}}(x).
	\end{align*} 
	To seek for $\hat{s}_{\bm l_0},$ as before we multiply both sides of \eqref{eq:biharmonic} by $\overline{ \phi_{\bm{l}_0}}$ and integrate by parts over $B_R$ which gives
	\begin{align*}
		&\frac{1}{a^2}\int_{\Gamma_R}\left(\partial_\nu\Delta u(x, k_{\bm l_0})+\ii\frac{2\pi}{a}({\bm{l}_0\cdot \nu})\Delta u(x, k_{\bm l_0})\right)\overline{ \phi_{\bm{l}_0}}\dd s(x)\\
		&\quad
		-\frac{k^2}{a^2}\int_{\Gamma_R}\left(\partial_\nu u(x, k_{\bm l_0})+\ii\frac{2\pi}{a}({\bm{l}_0\cdot \nu}) u(x, k_{\bm l_0})\right)\overline{ \phi_{\bm{l}_0}}\dd s(x)\\
		=&\frac{1}{a^2}\int_{V_0}\hat{s}_{\bm{l}_0}\overline{ \phi_{\bm{l}_0}(x)}\dd x+\frac{1}{a^2}\sum_{|\bm{l}|_\infty\ge1}\hat{s}_{\bm{l}}\int_{V_0} \phi_{\bm{l}}(x)\overline{ \phi_{\bm{l}_0}(x)}\dd x\\
		=&\frac{\sin\lambda \pi}{\lambda\pi}\hat{s}_{\bm{l}_0}+\frac{1}{a^2}\sum_{|{\bm{l}}|_\infty\ge1}\hat{s}_{\bm{l}}\int_{V_0} \phi_{\bm{l}}(x)\overline{ \phi_{\bm{l}_0}(x)}\dd x.
	\end{align*}
	As a consequence, we obtain the computation formula for $\hat{s}_{\bm{l}_0}$ as follows:
	\begin{align}
		\nonumber \hat{s}_{\bm l_0}=&\frac{\lambda\pi}{a^2\sin\lambda\pi}\int_{\Gamma_R}\left(\partial_\nu\Delta u(x, k_{\bm l_0})+\ii\frac{2\pi}{a}({\bm{l}}_0\cdot\nu)\Delta u(x, k_{\bm l_0})\right)\overline{ \phi_{\bm{l}_0}}\dd s(x)
		\\&-\frac{\lambda\pi k^2}{a^2\sin\lambda\pi}\int_{\Gamma_R}\left(\partial_\nu u(x, k_{\bm l_0})+\ii\frac{2\pi}{a}({\bm{l}}_0\cdot\nu) u(x, k_{\bm l_0})\right)\overline{ \phi_{\bm{l}_0}} \dd s(x)\nonumber\\
		&-\frac{\lambda\pi}{a^2\sin\lambda\pi}\sum_{1\le|{\bm l}|_\infty\le N}\hat{s}_{\bm l}\int_{V_0} \phi_{\bm l}(x)\overline{ \phi_{\bm{l}_0}(x)}\dd x.
		&\label{eq:s0}
	\end{align}
	We take $\hat{s}_{\bm l_0}$ as an approximation of $\hat{s}_{\bm 0}$.
	
	\subsection{A stable substitution technique}\label{stable}
	
	In practice, it is usually the case that we are only given the Dirichlet boundary measurements $u(x, k)|_{\Gamma_R}$ and $\Delta u(x, k)|_{\Gamma_R}$. However, to utilize the integral identities (\ref{eq:s_l}) and (\ref{eq:s0}), the normal derivatives $\partial_\nu u(x, k)|_{\Gamma_R}$ and $\partial_\nu \Delta u(x, k)|_{\Gamma_R}$ are also required. Additionally, as suggested in \cite{IP15} for the acoustic wave equation, we shall reconstruct the Fourier
	coefficients $\hat{s}_{\bm{l}}, \bm{l}\in\mathbb{Z}^2$ using the boundary data on $\Gamma_{\rho}$ for $\rho>R$. To this end, analogous to the analysis in Section~\ref{Fourier}, the Fourier coefficients $\hat{s}_{\bm{l}}$ can also be calculated through the following integral identity for $\rho>R$:
	\begin{align}\nonumber
		\hat{s}_{\bm{l}}=&\frac{1}{a^2}\int_{\Gamma_\rho}\left(\partial_\nu\Delta u+\ii\frac{2\pi}{a}({\bm l}\cdot\nu)\Delta u\right)\overline{ \phi_{\bm{l}}(x)}\mathrm{d}s(x)\\
		\label{eq:s_l1}&-\frac{k^2}{a^2}\int_{\Gamma_\rho}\left(\partial_\nu u+\ii\frac{2\pi}{a}({\bm l}\cdot\nu) u\right)\overline{ \phi_{\bm{l}}(x)}\mathrm{d}s(x),\ \ 
		{\bm l}\in\mathbb{Z}^2\setminus\bm 0,
	\end{align}
	and
	\begin{align}
		\nonumber  \hat{s}_{\bm{l}_0}=&\frac{\lambda\pi}{a^2\sin\lambda\pi}\int_{\Gamma_\rho}\left(\partial_\nu\Delta u+\ii\frac{2\pi}{a}({\bm{l}}_0\cdot\nu)\Delta u\right)\overline{ \phi_{\bm{l}_0}}\dd s(x)
		\\&-\frac{\lambda\pi k^2}{a^2\sin\lambda\pi}\int_{\Gamma_\rho}\left(\partial_\nu u+\ii\frac{2\pi}{a}({\bm{l}}_0\cdot\nu) u\right)\overline{ \phi_{\bm{l}_0}} \dd s(x)\nonumber\\
		&-\frac{\lambda\pi}{a^2\sin\lambda\pi}\sum_{1\le|{\bm l}|_\infty\le N}\hat{s}_{\bm l}\int_{V_0} \phi_{\bm l}(x)\overline{ \phi_{\bm{l}_0}(x)}\dd x.
		&\label{eq:s01}
	\end{align}
	
	Given the measurements $u(x, k)|_{\Gamma_R}$ and $\Delta u(x, k)|_{\Gamma_R}$, in what follows we show that the boundary terms $u(x, k)|_{\Gamma_\rho}$, $\Delta u(x, k)|_{\Gamma_\rho}$, $\partial_\nu u(x, k)|_{\Gamma_\rho}$ and $\partial_\nu \Delta u(x, k)|_{\Gamma_\rho}$ in the integrals (\ref{eq:s_l1})-(\ref{eq:s01}) can be obtained via the series expansion of the solutions exterior to $\Gamma_R$. To this end,
	introduce two auxiliary functions
	\begin{align}\label{eq:vHvM}
		u_H=-\frac{1}{2k^2}(\Delta u-k^2u),\quad u_M=\frac{1}{2k^2}(\Delta u+k^2u).
	\end{align}
	Then we decompose $u$ and $\Delta u$ in the exterior domain $\mathbb{R}^2\backslash\overline{B_R}$ by
	\begin{align}\label{eq:vDeltav}
		u=u_H+u_M,\quad \Delta u=k^2(u_M-u_H),
	\end{align}
	where $u_H$ and $u_M$ satisfy the following Helmholtz and modified Helmholtz equations
	\begin{align}
		\begin{cases}
			\Delta u_H+k^2u_H=0 \quad\mbox{in}\quad \mathbb{R}^2\backslash\overline{B_R},\cr
			\lim_{r\to\infty}\sqrt{r} (\partial_r u_H-\mathrm{i}k u_H)=0,
		\end{cases} \quad
		\begin{cases}
			\Delta u_M-k^2u_M=0 \quad\mbox{in}\quad \mathbb{R}^2\backslash\overline{B_R},\cr
			\lim_{r\to\infty}\sqrt{r} (\partial_r u_M+\mathrm{i}k u_M)=0.
		\end{cases}
	\end{align}
	Based on the above decomposition, to compute $u, \Delta u$ and $\partial_\nu u, \partial_\nu\Delta u$ on $\Gamma_\rho$
	it suffices to compute $u_\xi(x, k)|_{\Gamma_\rho}$, $\partial_\nu u_\xi(x, k)|_{\Gamma_\rho}$, $\xi\in\{H,M\}$. Using the polar coordinates $(r,\theta):x=r(\cos\theta,\sin\theta)$, the solutions $u_H$ and $u_M$ admits 
	the following series expansions for $r\ge R$,
	\begin{align*}
		u_H(x, k)=\sum_{n\in\mathbb{Z}}\frac{H_n^{(1)}(kr)}{H_n^{(1)}(kR)}\hat{u}^{H}_{k,n}\mathrm{e}^{\ii n\theta},\ \
		u_M(x, k)=\sum_{n\in\mathbb{Z}}\frac{K_n(kr)}{K_n(kR)}\hat{u}^{M}_{k,n}\mathrm{e}^{\ii n\theta},
	\end{align*}
	where $H_n^{(1)}$ and $K_n$ denote the first-kind Hankel function and second-kind modified Bessel function, respectively, with order $n$ and the Fourier coefficients $\hat{u}^{H}_{k,n}$ and $\hat{u}^{M}_{k, n}$ can be obtained using the boundary data
	$u_H(x, k)|_{\Gamma_R}$ and $u_M(x, k)|_{\Gamma_R}$ as follows
	\begin{align*}
		\hat{u}^\xi_{k, n} & =\frac{1}{2\pi}\int_0^{2\pi}u_\xi(R,\theta;k)\mathrm{e}^{-\ii n\theta}\mathrm{d}\theta,\quad \xi\in\{H,M\}.
	\end{align*}
	Therefore, for $x\in\Gamma_\rho$ we have
	\begin{align}\label{eq:vHvM_expansion}
		u_H(x, k)=\sum_{n\in\mathbb{Z}}\frac{H_n^{(1)}(k\rho)}{H_n^{(1)}(kR)}\hat{u}^{H}_{k,n}\mathrm{e}^{\ii n\theta},\ \
		u_M(x, k)=\sum_{n\in\mathbb{Z}}\frac{K_n(k\rho)}{K_n(kR)}\hat{u}^{M}_{k,n}\mathrm{e}^{\ii n\theta},
	\end{align}
	and
	\begin{align}\label{eq:partialvHvM_expansion}
		\partial_{\nu}u_H(x, k)=\sum_{n\in\mathbb{Z}}k\frac{{H_n^{(1)}}'(k\rho)}{H_n^{(1)}(kR)}\hat{u}^{H}_{k,n}\mathrm{e}^{\ii n\theta},\ \
		\partial_{\nu}u_M(x, k)=\sum_{n\in\mathbb{Z}}k\frac{{K_n}'(k\rho)}{K_n(kR)}\hat{u}^{M}_{k,n}\mathrm{e}^{\ii n\theta}.
	\end{align}
	Thus, from \eqref{eq:vDeltav} we have that the boundary values of $u(x, k)$, $\Delta u(x, k)$, $\partial_\nu u(x, k)$ and $\partial_\nu \Delta u(x, k)$
	on $\Gamma_\rho$
	can be computed by the Dirichlet data $u(x, k)$ and $\Delta u(x, k)$ on $\Gamma_R$.
	
	
	\section{Numerical examples}\label{sec:numerical}
	In this section, we present several numerical examples to verify the performance of the proposed Fourier method. The synthetic data are generated via the integration. For $S\in L^2(\mathbb{R}^2)$ with $\text{supp }S\subset V_0,$ the unique solution to \eqref{eq:biharmonic} subject to the radiation condition \eqref{src} is given by \eqref{eq:solution1}. 
	For the numerical implementation of the Fourier method, the computation domain is chosen to be $V_0=[-0.5,0.5]\times[-0.5,0.5]$ in the first two examples and $V_0 = [-3,3]\times[-3,3]$ in the last example. The volume integral over $V_0$ is evaluated over a $201\times 201$ grid of uniformly spaced points $x_m\in V_0,\,m=1,\cdots, 201^2.$
	Once the radiated data is obtained, we further disturb the data by
	\begin{align*}
		u^\delta&:=u+\delta r_1|u|\mathrm{e}^{\ii\pi r_2}, \\
		\Delta u^\delta & :=\Delta u+\delta r_3|u|\mathrm{e}^{\ii\pi r_4},
	\end{align*}
	where $r_\ell,\ell=1,\cdots,4$ are uniformly distributed random numbers ranging from  $-1$ to 1, and $\delta\in[0,1)$ is the noise level.
	As a rule of thumb, without otherwise specified, we choose the truncation $N$ by
	\begin{align}\label{eq:N}
		N:=5\left[\delta^{-1/4}\right],
	\end{align}
	where $[X]$ denotes the largest integer that is smaller than $X+1.$
	Once $N=N(\delta)$ is given, the wavenumber $k_{\bm l}$ for ${\bm l}\in\mathbb{Z}^2$ is given by
	\begin{align*}
		k_{\bm l}  =\left\{
		\begin{aligned}
			&\frac{10\pi}{3}|{\bm l}|:{\bm l}\in\mathbb{Z}^2,1\le|{\bm l}|_\infty\le N,\\
			&\frac{10\pi}{3}\lambda,\quad|{\bm l}|_\infty=0,
		\end{aligned}
		\right.
	\end{align*}
	with $\lambda=10^{-3}.$
	Given the admissible wavenumbers, we measure the radiated data
	\begin{align*}
		\left\{
		u(k_i,R,\theta_j):k_i\in\mathbb{K}\cup\{k_{\bm 0}\},i=1,\cdots,(2N+1)^2,\theta_j\in[0,\theta_{\max}],j=1,\cdots, N_m
		\right\}
	\end{align*}
	on the measurement circle centered at the origin with $R=0.8$ for the first two examples and $R=5$ for the last example. $N_m=200$ equally spaced measurement angles $\theta_j=j\theta_{\max}/N_m$ where $\theta_{\max }$ is the measurement aperture.
	
	To evaluate the approximate Fourier coefficients defined in \eqref{eq:s_l1} and \eqref{eq:s01}, we calculate
	\begin{align*}
		\left\{\mathcal{M}(\rho,\Theta_j; k_i):k_i\in  \mathbb{K}\cup\{k_{\bm 0}\}:i=1,\cdots,(2N+1)^2,
		\Theta_j\in [0,2\pi]:j=1,\cdots, 200 \right\},
	\end{align*}
	where $\mathcal{M}$ stands for $u^\delta,\,\Delta u^\delta,\,\partial_{\nu_\rho}u^\delta$ and $\partial_{\nu}\Delta u^\delta,$ respectively.
	Here, we take $\Theta_j=\frac{j\pi}{100}.$ In the first two examples, $\rho=2.$ In the last example, $\rho=6.$  To approximate the infinite series in \eqref{eq:vHvM_expansion} and \eqref{eq:partialvHvM_expansion}, we numerically truncate them by $|n|\le 60.$
	Under these parameters, the approximate source function $S_N(x)$ is computed by \eqref{eq:SN}.
	
	To evaluate the performance qualitatively, we compute the discrete relative $L^2$ error:
	\begin{align*}
		\frac{\|S_N^\delta-S\|_{0}}{\|S\|_{0}}:=\frac{\left(\sum_{m=1}^{201^2}\left|S_N^\delta(x_m)-S(x_m)\right|^2\right)^{1/2}}{\left(\sum_{m=1}^{201^2}\left|S(x_m)\right|^2\right)^{1/2}},
	\end{align*}
	and $H^1$ error:
	\begin{align*}
		\frac{\|S_N^\delta-S\|_1}{\|S\|_1}:=\frac{\left(\sum_{m=1}^{201^2}\left|\nabla S_N^\delta(x_m)-\nabla S(x_m)\right|^2+\left|S_N^\delta(x_m)-S(x_m)\right|^2\right)^
			{1/2}}{\left(\sum_{m=1}^{201^2}\left|\nabla S(x_m)\right|^2+\left|S(x_m)\right|^2\right)^{1/2}},
	\end{align*}
	respectively.
	\begin{example}\label{exa1}
		In the first numerical example, we consider the reconstruction of a smooth function of mountain-shape, which is given by
		\begin{align*}
			S_1(x_1,x_2)=1.1\mathrm{e}^{-200\left((x_1-0.01)^2+(x_2-0.12)^2\right)}-100\left(x_2^2-x_1^2\right)\mathrm{e}^{-90\left(x_1^2-x_2^2\right)}.
		\end{align*}
		In \Cref{fig:S1_exact}, we plot the surface and the contour of the exact source function. To compare the accuracy of the reconstruction, we list the relative $L^2$ and $H^1$ errors in \Cref{tab:error} under different noise levels. In addition, we also display the CPU time of the inversion in the last row of \Cref{tab:error} (All the computations in the numerical experiments are run on an Intel Core 2.6 GHz laptop). By taking $\delta$ to be $10\%,$ we exhibit the reconstruction of the source in \Cref{fig:S1_reconstruction}. From  \Cref{fig:S1_exact} and \Cref{tab:error}, we can see that the source function can be well-reconstructed under different noise levels. Comparing different columns in \Cref{tab:error}, we find that our method is stable and insensitive to random noise.
		
		Though there exists an alternative choice of the truncation $N$ in \eqref{eq:N}, it is $N$ that mainly influences the effect of reconstruction in our method. If $N$ is not chosen properly, the reconstruction may be unsatisfactory. To illustrate the influence of the truncation, we take 
		\begin{align}
			\label{eq:N2}N=2[\delta^{-1/3}],
		\end{align}
as done in \cite{IP15} and reconsider the reconstruction the results obtained in \Cref{fig:S1_reconstruction} with other parameters unchanged. The reconstruction is shown in \Cref{fig:S1_N}. As shown in \Cref{fig:S1_N}, the reconstruction is destructed due to the inappropriate choice of $N.$ In this case, the relative errors of the reconstruction are
		\begin{align*}
			\frac{\|S_1^\delta-S_1\|_0}{\|S_1\|_0} =10.46\%,\ \
			\frac{\|S_1^\delta-S_1\|_1}{\|S_1\|_1} =18.87\%,
		\end{align*}
which are much larger than the errors with $N$ defined in \eqref{eq:N}. Thus, it is of vital importance to choose a proper truncation $N$ of the Fourier expansion. Since an optimal choice of $N$ is still open, we would like to take $N$ as defined in \eqref{eq:N} in the following numerical experiments.
		
	\begin{figure}[htp]
			\centering
			\subfigure[]{\includegraphics[width=0.48\textwidth]{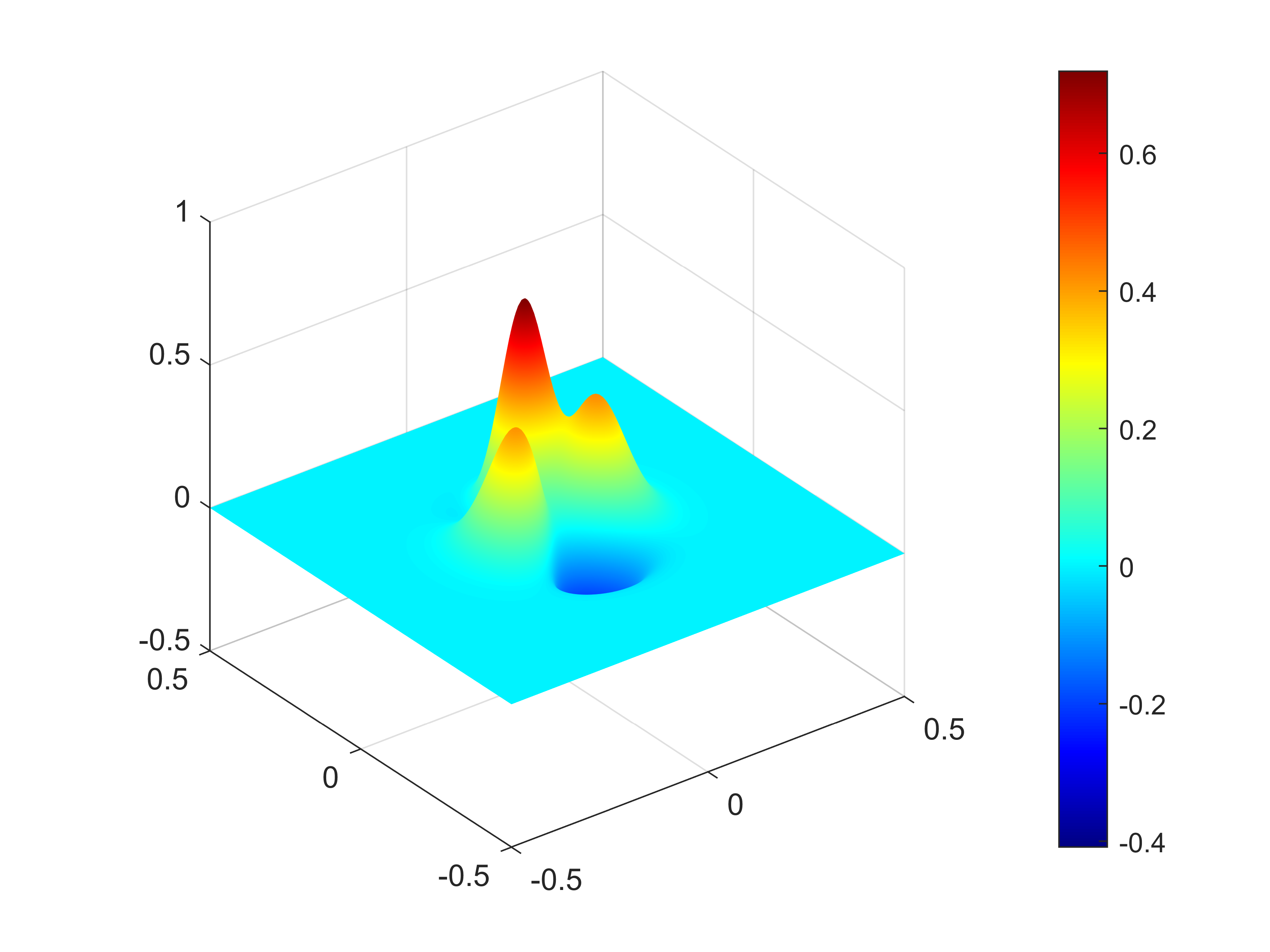}}
			\subfigure[]{\includegraphics[width=0.48\textwidth]{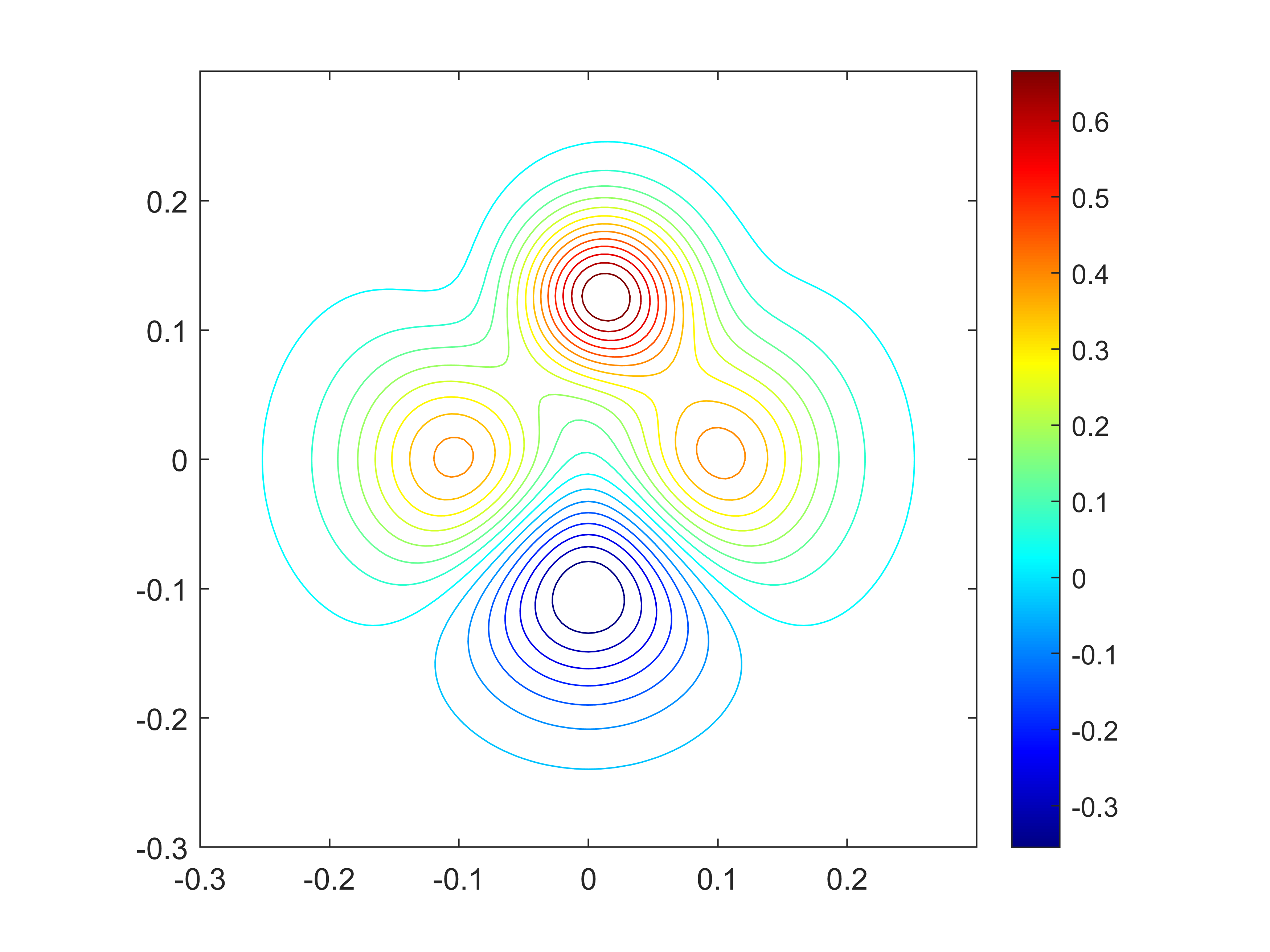}}
			\caption{The exact source function $S_1.$ (a) Surface plot; (b) contour plot.}\label{fig:S1_exact}
	\end{figure}  
		
		\begin{figure}[htp]
			\centering
			\subfigure[]{\includegraphics[width=0.48\textwidth]{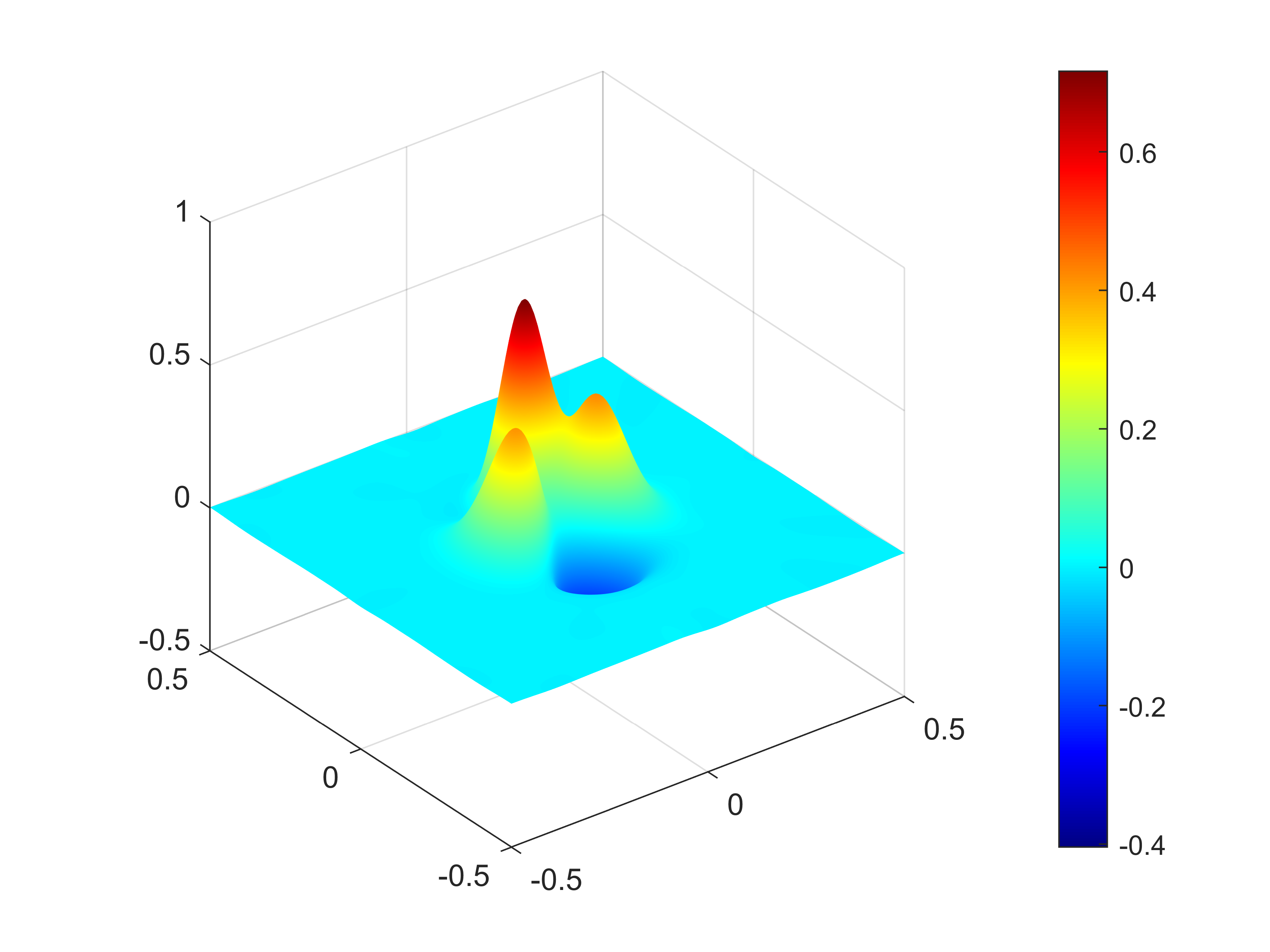}}
			\subfigure[]{\includegraphics[width=0.48\textwidth]{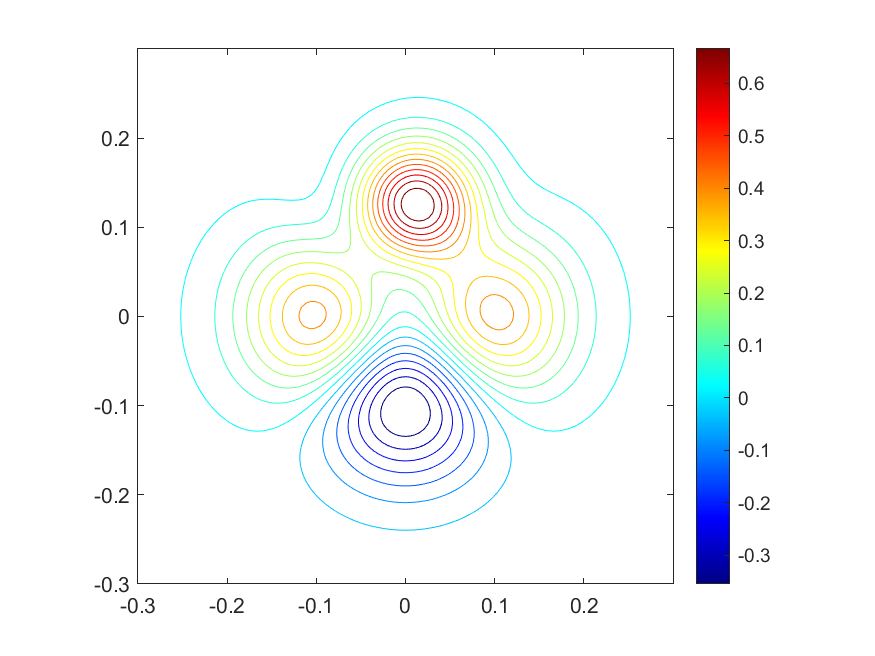}}
			\caption{The reconstruction of source function $S_1$ with $\delta=10\%$. (a) Surface plot; (b) contour plot.}\label{fig:S1_reconstruction}
		\end{figure}
		\begin{figure}[htp]
			\centering
			\subfigure[]{\includegraphics[width=0.48\textwidth]{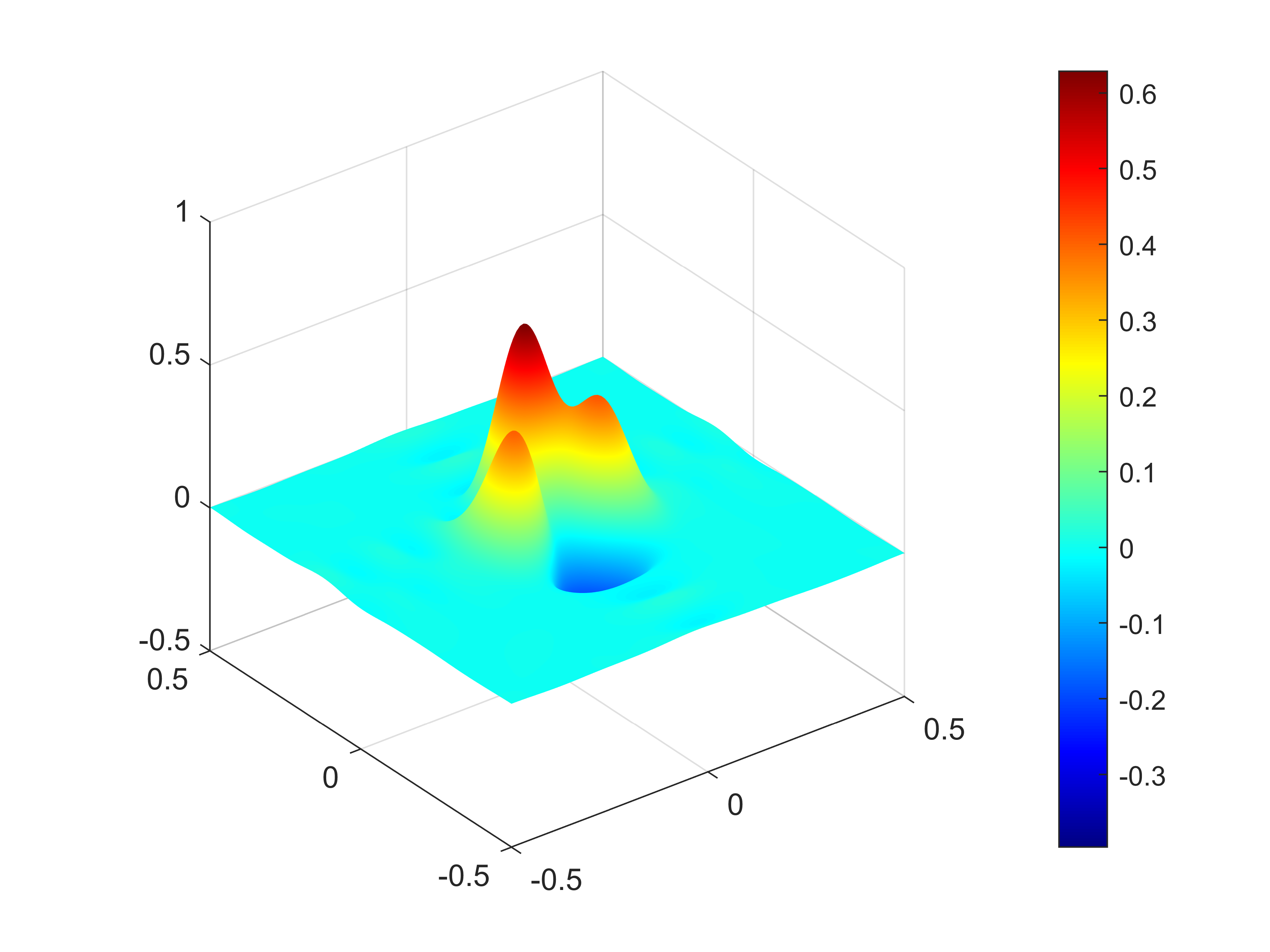}}
			\subfigure[]{\includegraphics[width=0.48\textwidth]{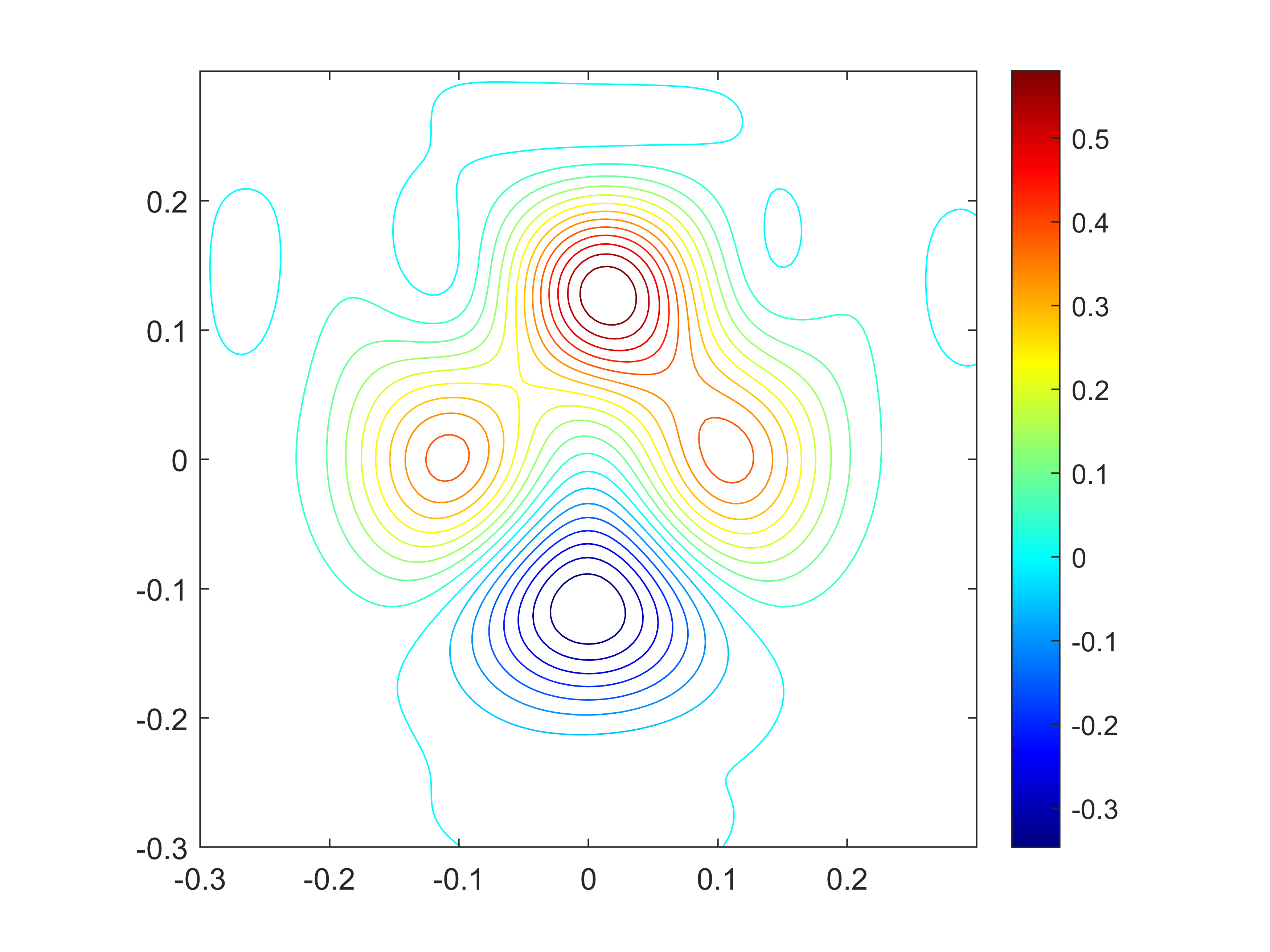}}
			\caption{The reconstruction of source function $S_1$ with $\delta=10\%$ and $N$ defined by \eqref{eq:N2}. (a) Surface plot; (b) contour plot.}\label{fig:S1_N}
		\end{figure}
		
		\begin{table}
			\centering
			\begin{tabular}{lccccc}
				\toprule
				$\delta$& $0.5\%$ &  $5\%$ & $10\%$&$20\%$ \\\hline
				$N(\delta)$ & 20  & 15 & 10 & 10\\
				$\|S_1^\delta-S_1\|_0/\|S_1\|_0$ &$2.2012\%$ &$2.2018\% $ & $ 2.2149\%$&$2.2165\%$ \\
				$\|S_1^\delta-S_1\|_1/\|S_1\|_1$ &$2.4917\%$ & $2.4918\% $ & $2.5690\%$&$2.5697\%$ \\
				Time (second) &$1.76$  &$1.72 $ & $1.07 $&$1.13$ \\\bottomrule
				\hline
			\end{tabular}
			\caption{The reconstruction error of $S_1$ and computing time subject to different $\delta.$}\label{tab:error}
		\end{table}
	\end{example}
	
	\begin{example}\label{exa2}
		In the second example, we consider the reconstruction of a discontinuous source function. The exact source function is given by
		\begin{equation*}
			S_2(x_1,x_2)  =
			\begin{cases}
				0.8 & \text{if }x_1^2+x_2^2<0.04,\\
				0.3 & \text{if }0.04\le x_1^2+x_2^2\le 0.09,\\
				0 & \text{elsewhere,}
			\end{cases}
		\end{equation*}
	which is plotted in \Cref{fig:S2_exact}. By taking $\theta_{\max}=2\pi,$  we exhibit the reconstruction in \Cref{fig:S2_2pi}. Comparing \Cref{fig:S2_exact} and \Cref{fig:S2_2pi}, we can see that the source function can be roughly reconstructed. To better understand the reconstruction, we depict several cross-sections of the source function and the reconstruction in \Cref{fig:S2_cross}. From \Cref{fig:S2_cross}, we can find that our method exhibits distinct reconstruction performance by taking different $x_2$. Meanwhile, the jump discontinuity of the source function can be captured by the proposed method, and the well-known Gibbs phenomenon around the jumps is clearly observed.
	\end{example}
	
	\begin{figure}[htp]
		\centering
		\subfigure[]{\includegraphics[width=0.48\textwidth]{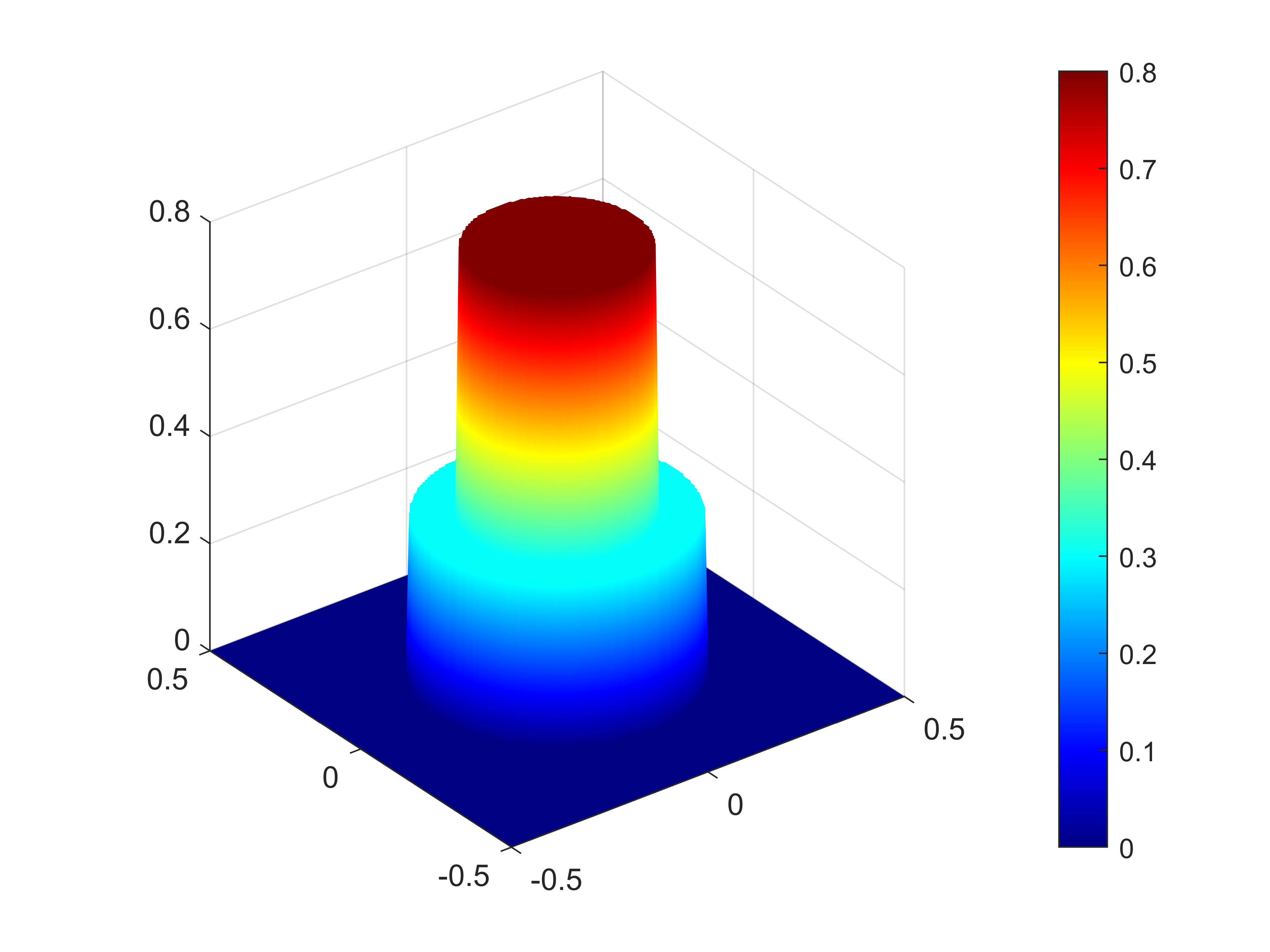}}
		\subfigure[]{\includegraphics[width=0.48\textwidth]{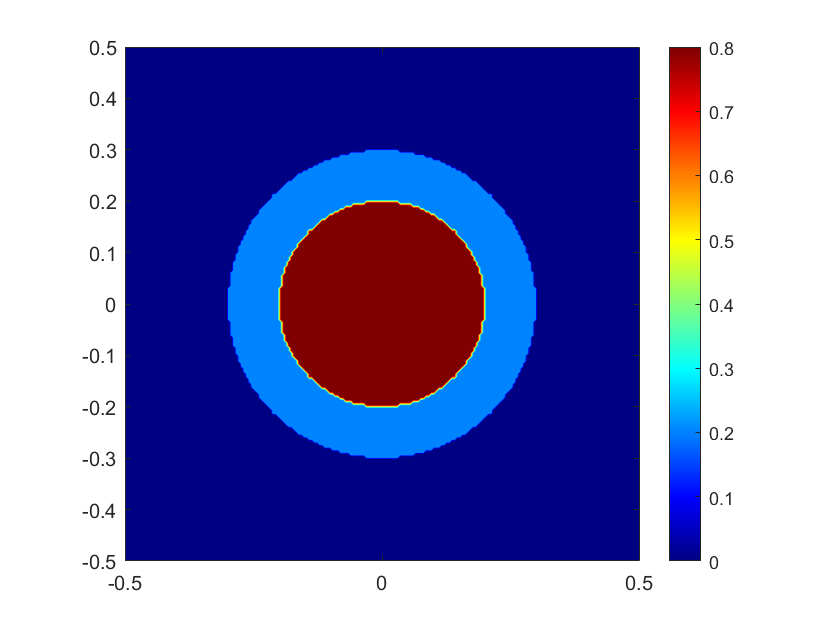}}
		\caption{The exact source function $S_2.$ (a) Surface plot; (b) contour plot.}\label{fig:S2_exact}
	\end{figure}

\begin{figure}[htp]
	\centering
	\subfigure[]{\includegraphics[width=0.48\textwidth]{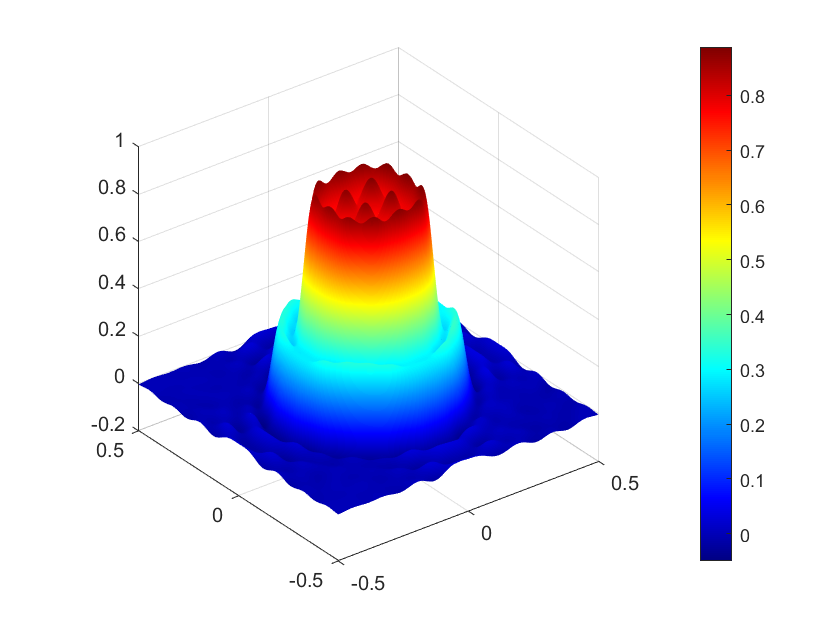}}
	\subfigure[]{\includegraphics[width=0.48\textwidth]{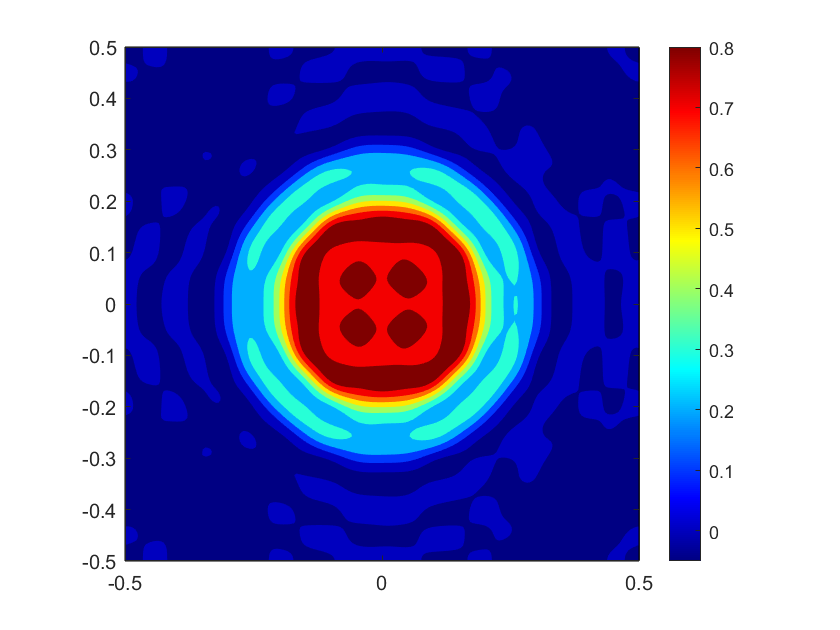}}
	\caption{Reconstruction of the source function $S_2$ with $\theta_{\max}=2\pi.$ (a) Surface plot; (b) contour plot.}\label{fig:S2_2pi}
\end{figure}

\begin{figure}[htp]
	\centering
	\subfigure[]{\includegraphics[width=0.48\textwidth]{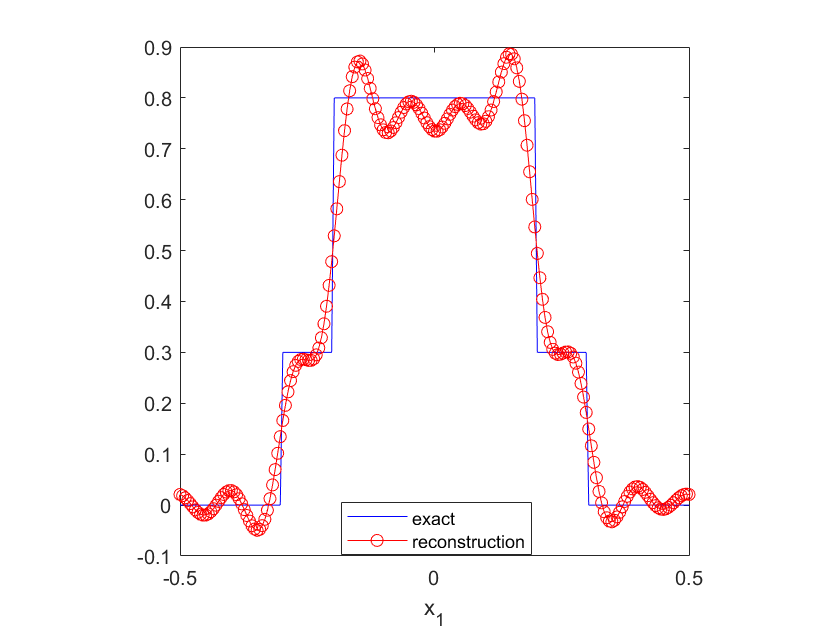}}
	\subfigure[]{\includegraphics[width=0.48\textwidth]{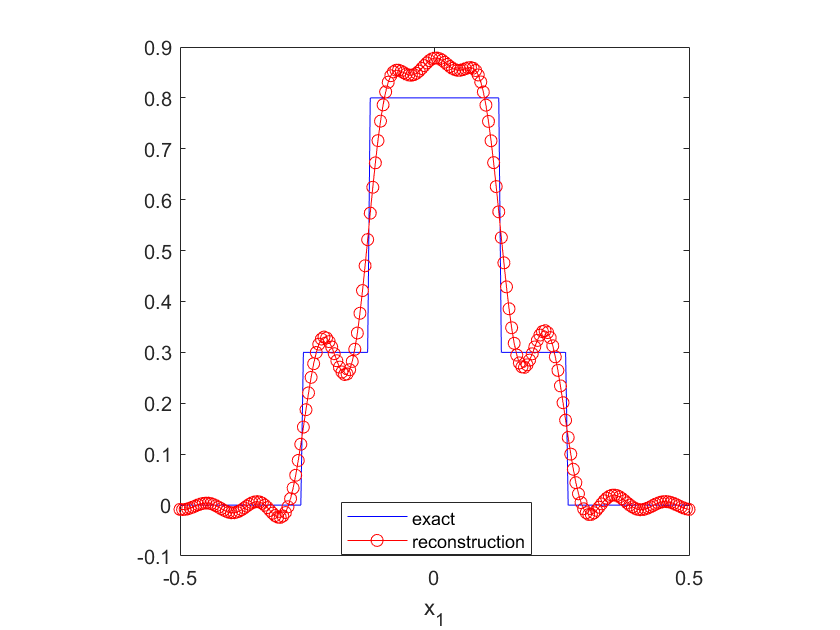}}
	\subfigure[]{\includegraphics[width=0.48\textwidth]{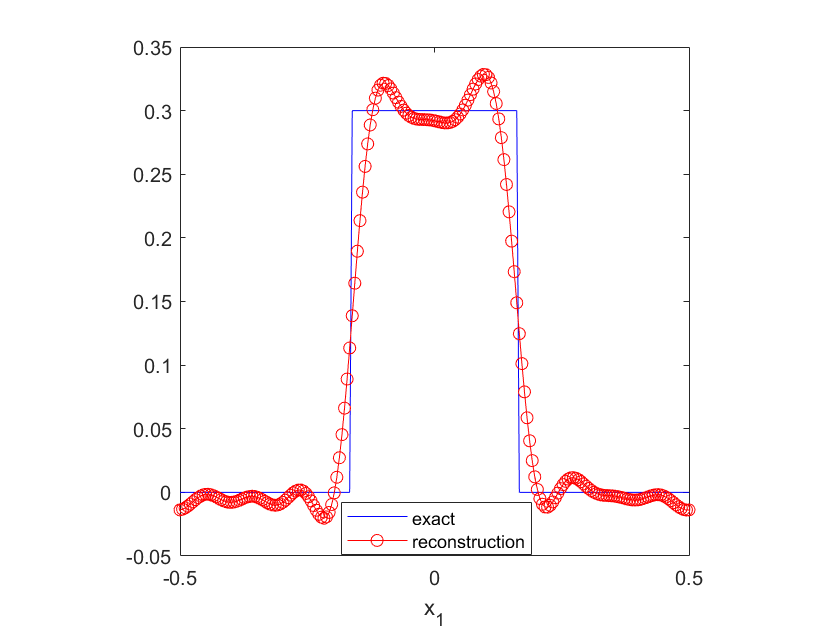}}
	\subfigure[]{\includegraphics[width=0.48\textwidth]{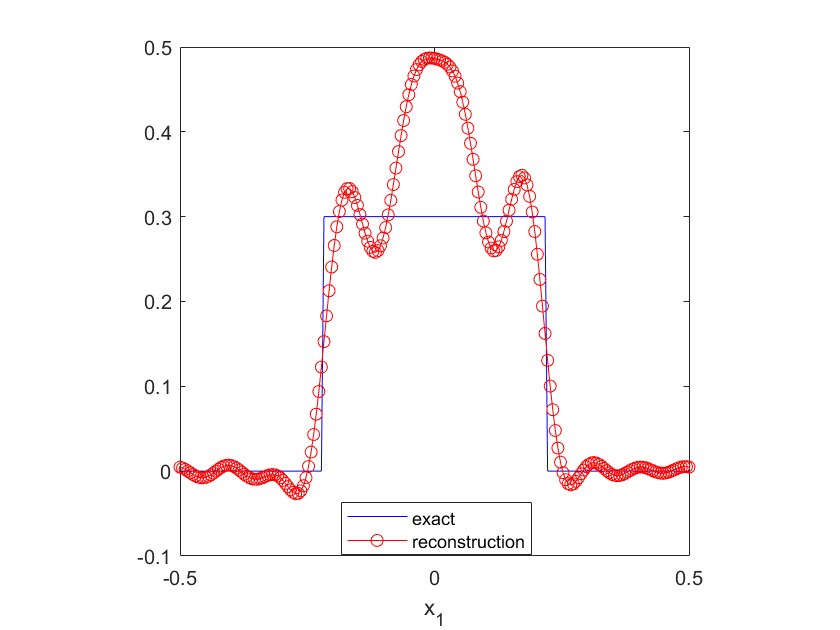}}
	\caption{Several cross-sections for the reconstruction of $S_2.$ (a) $x_2=0$; (b) $x_2=0.15$; (c) $x_2=-0.25$; (d) $x_2=-0.2$.}\label{fig:S2_cross}
\end{figure}

\begin{example}\label{exa3}
In the last example, we consider the reconstruction of the source function given by
\begin{align*}
			S_3(x_1,x_2)=0.3(1-x_1)^2\mathrm{e}^{-x_1^2-(x_2+1)^2}-(0.2x_1-x_1^3-x_2^5)\mathrm{e}^{-x_1^2-x_2^2}-0.03\mathrm{e}^{-(x_1+1)^2-x_2^2}.
\end{align*}

We display the exact source function in \Cref{fig:S3_exact}. In \Cref{fig: S3_reconstruction}, we present the reconstruction and the error under different observation apertures. Further, we depict several cross-section plots at $x_2=0$ and $x_2=-1.5,$ respectively in \Cref{fig: S3_reconstruction2}. From \Cref{fig: S3_reconstruction} and \Cref{fig: S3_reconstruction2}, we can see that using the full-aperture observation, the reconstructed source function matches well with the exact one and the reconstruction error is relatively small. When only limited-aperture data is available, the overall reconstruction is less satisfactory though some parts of the function can be reconstructed.
\end{example}
	
\begin{figure}[htp]
		\centering
		\subfigure[]{\includegraphics[width=0.48\textwidth]{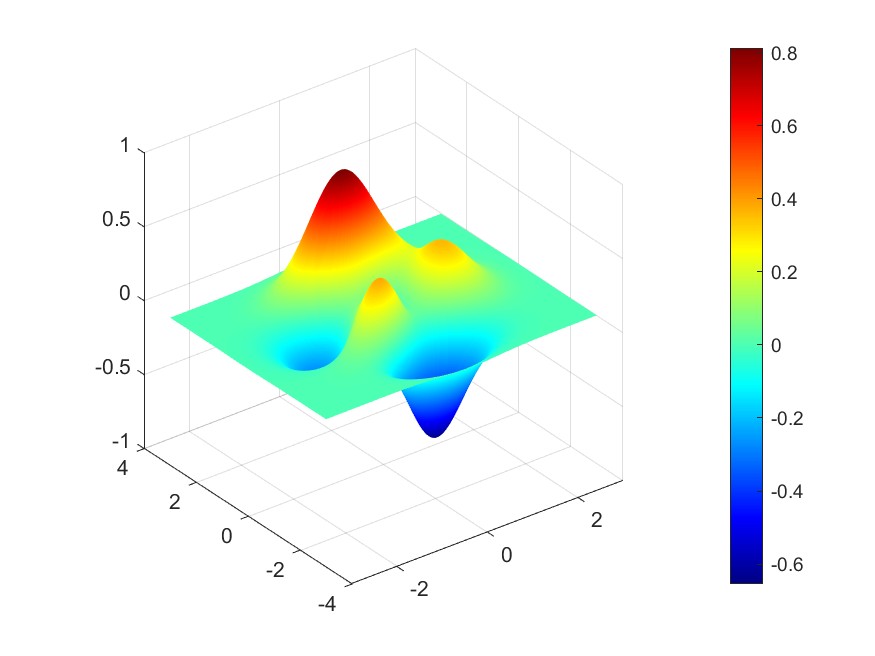}}
		\subfigure[]{\includegraphics[width=0.48\textwidth]{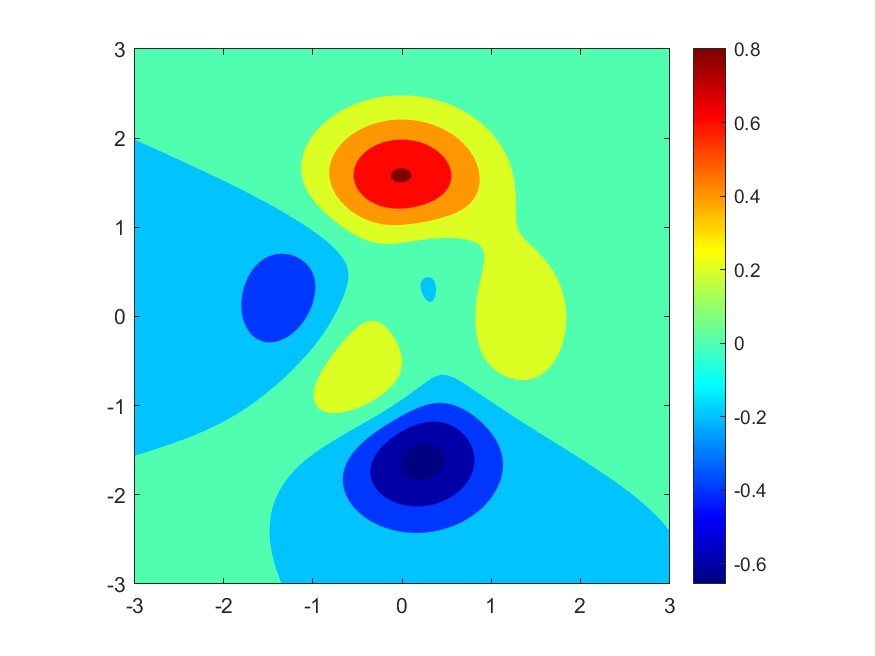}}
		\caption{The exact source function $S_3.$ (a) Surface plot; (b) contour plot.}\label{fig:S3_exact}
\end{figure}
	
\begin{figure}[htp]
		\subfigure[]{\includegraphics[width=0.45\textwidth]{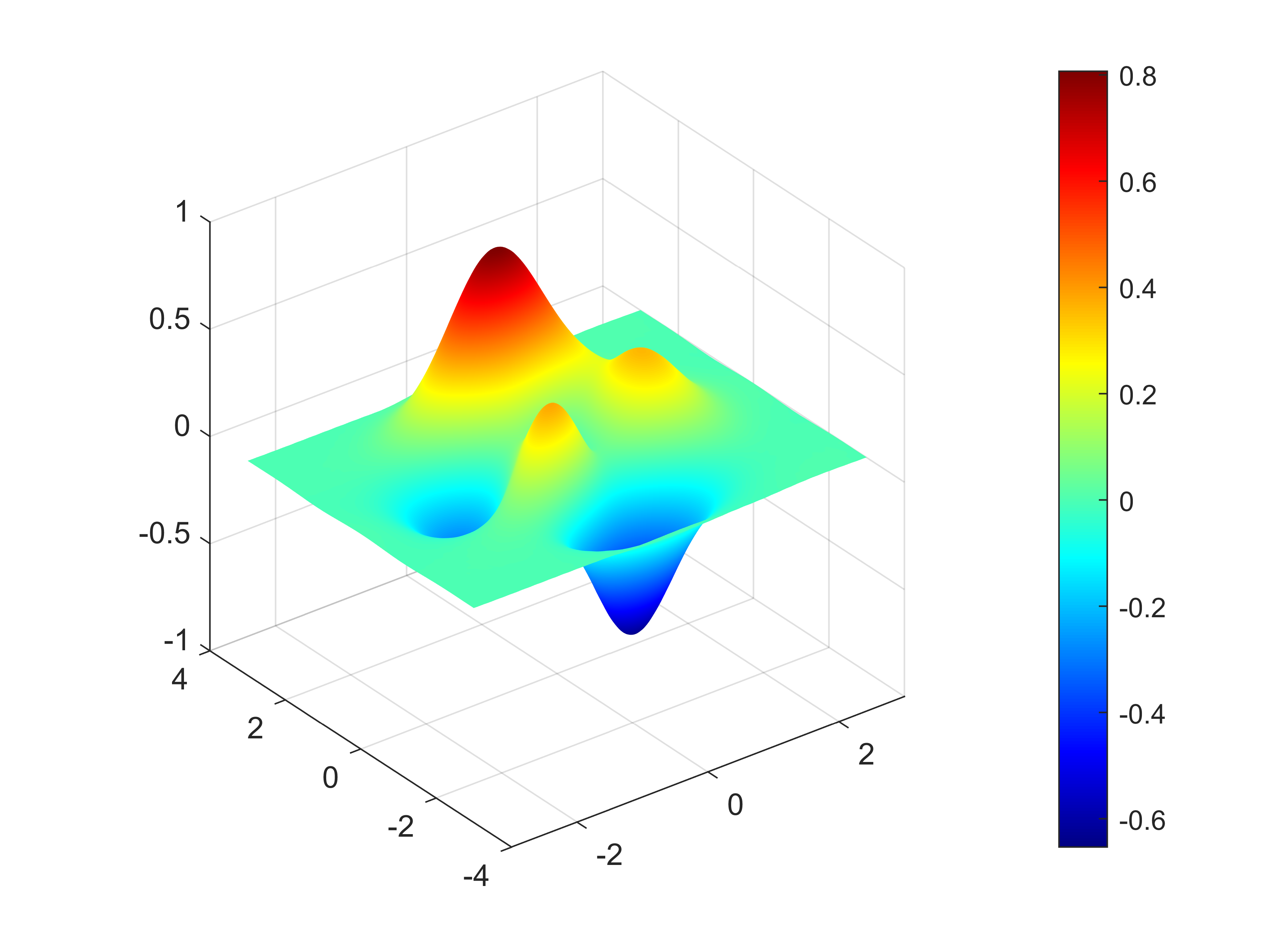}}
		\subfigure[]{\includegraphics[width=0.45\textwidth]{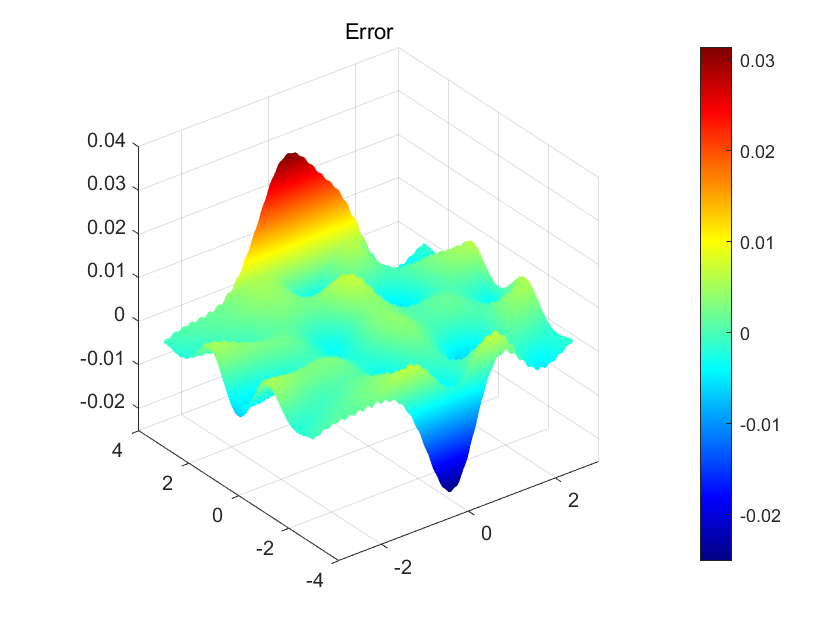}}
		\subfigure[]{\includegraphics[width=0.45\textwidth]{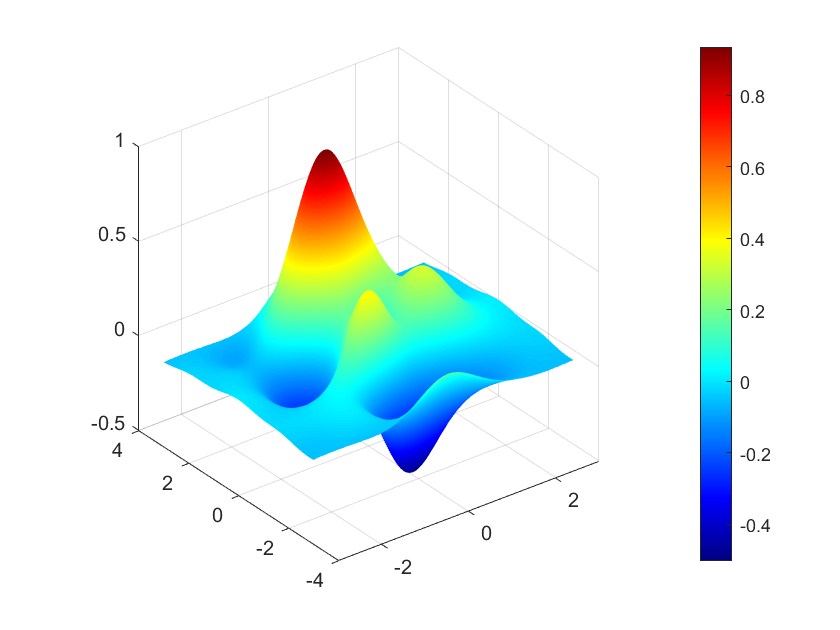}}
		\subfigure[]{\includegraphics[width=0.45\textwidth]{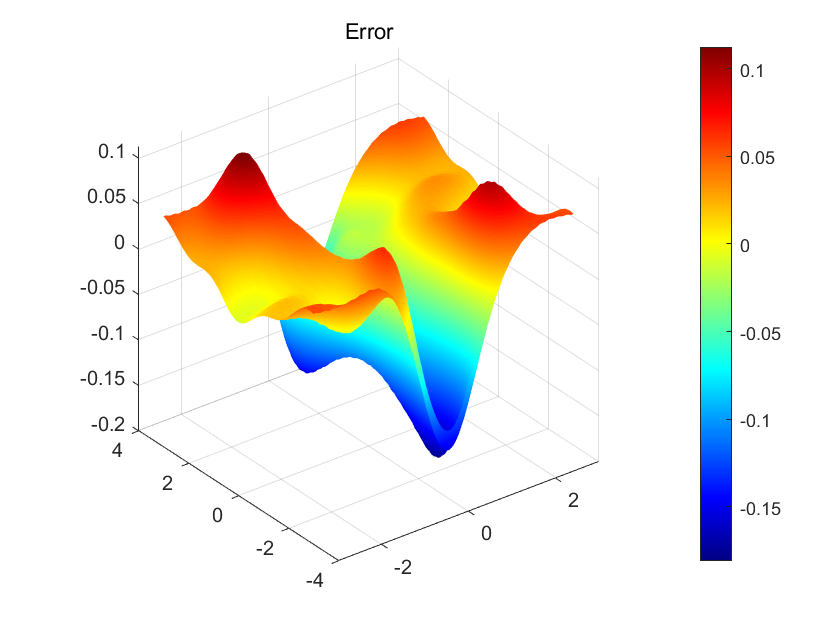}}
		\subfigure[]{\includegraphics[width=0.45\textwidth]{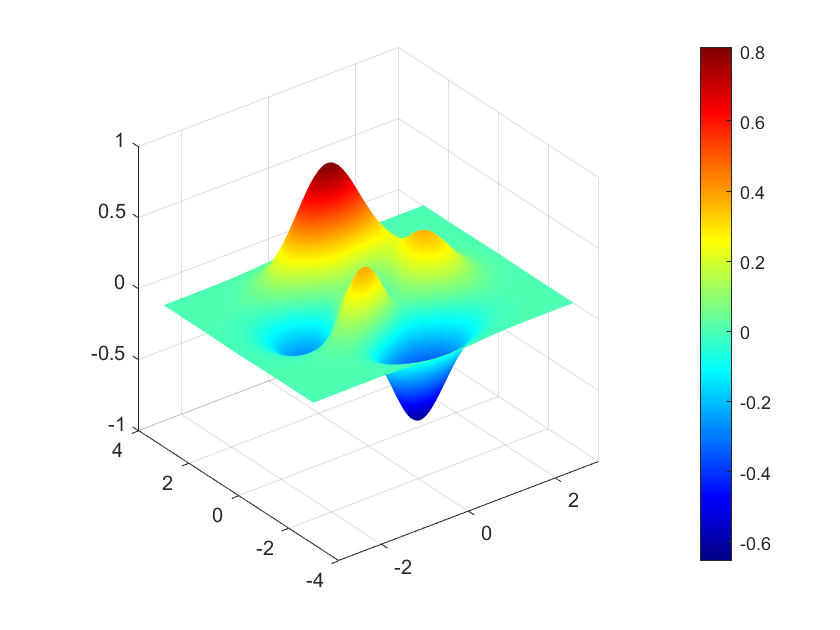}}\qquad\qquad\,
		\subfigure[]{\includegraphics[width=0.45\textwidth]{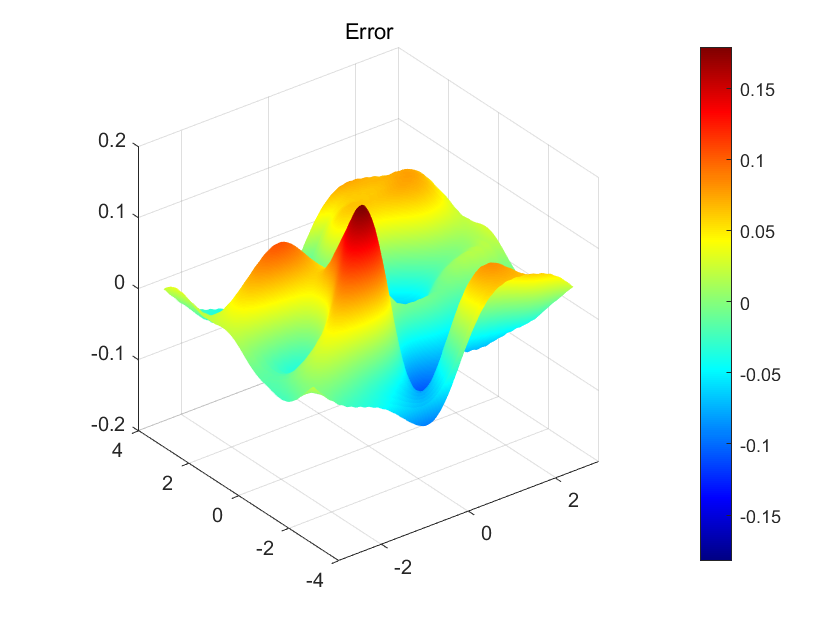}}
		\caption{The reconstruction of source function $S_3$ with different $\theta_{\max}.$ (a)--(b): reconstruction under full-aperture observation; (c)--(d): reconstruction with $\theta_{\max}=\frac{3\pi}{2};$ (e)--(f): reconstruction with $\theta_{\max}=\pi$. The left column: reconstructions; the right column: reconstruction errors.}\label{fig: S3_reconstruction}
\end{figure} 
	
\begin{figure}[htp]
		\subfigure[]{\includegraphics[width=0.45\textwidth]{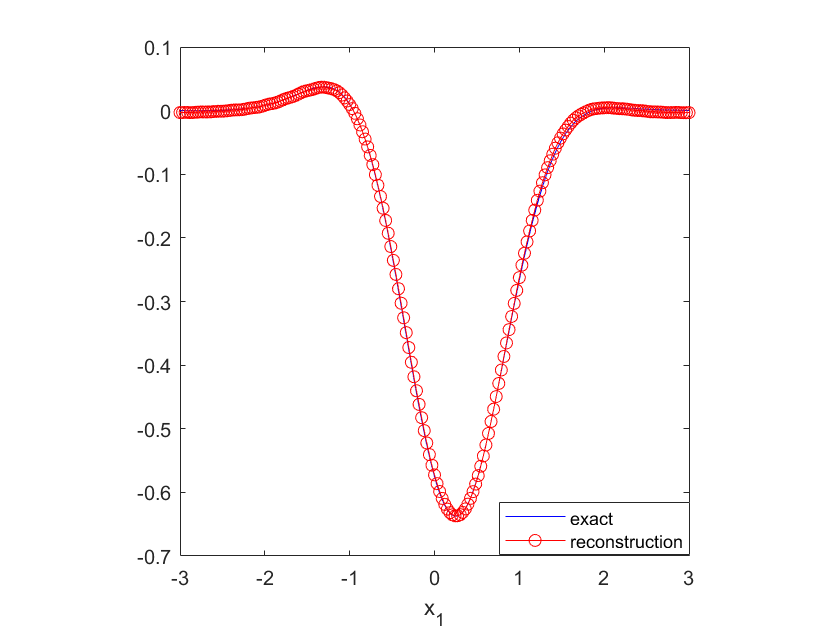}}
		\subfigure[]{\includegraphics[width=0.45\textwidth]{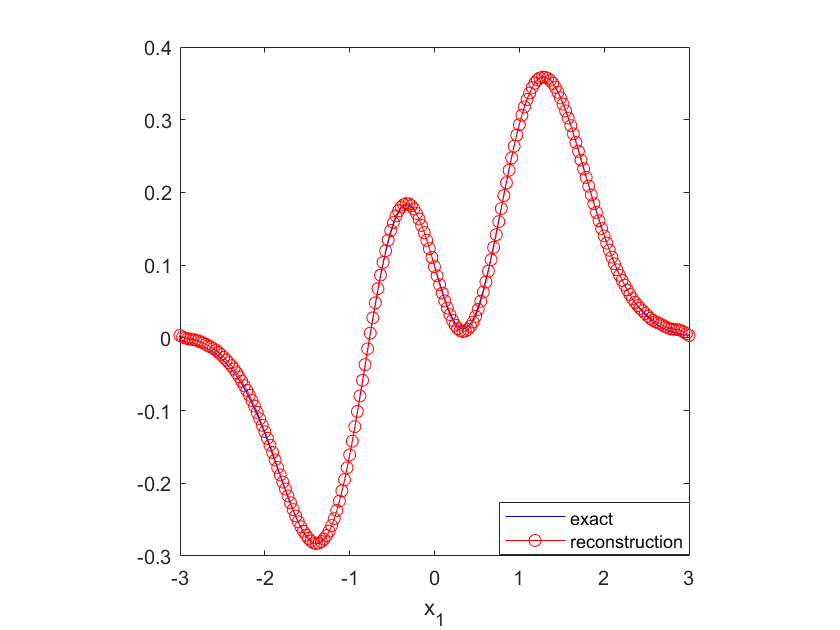}}\\
		\subfigure[]{\includegraphics[width=0.45\textwidth]{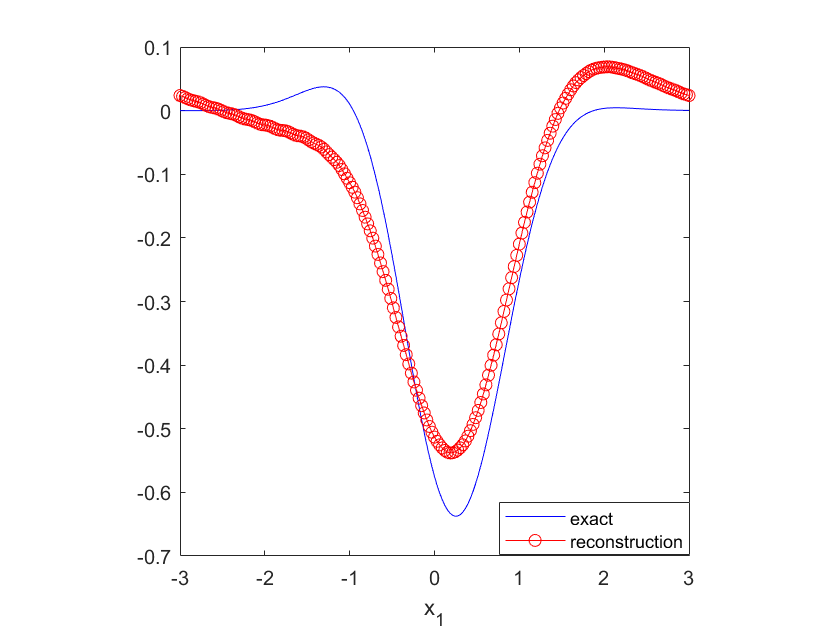}}
		\subfigure[]{\includegraphics[width=0.45\textwidth]{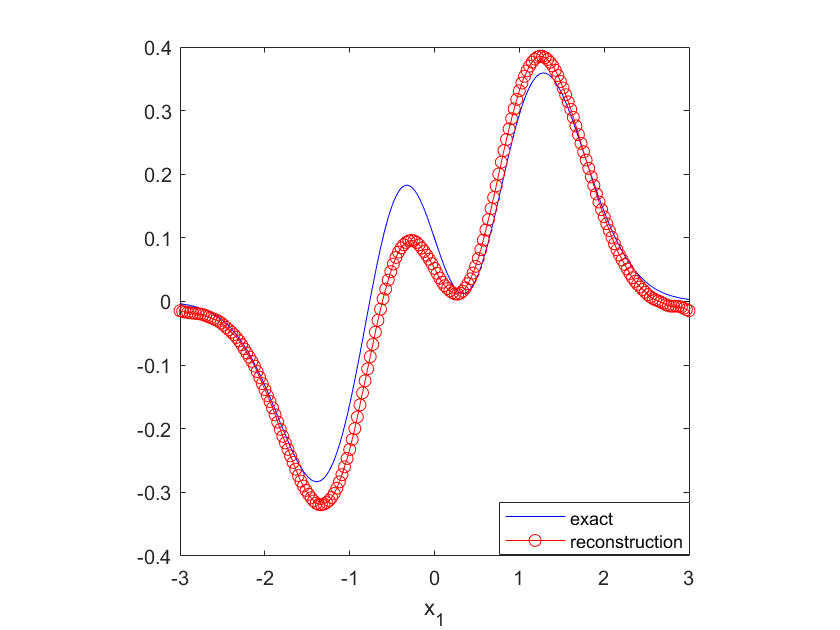}}\\
		\subfigure[]{\includegraphics[width=0.45\textwidth]{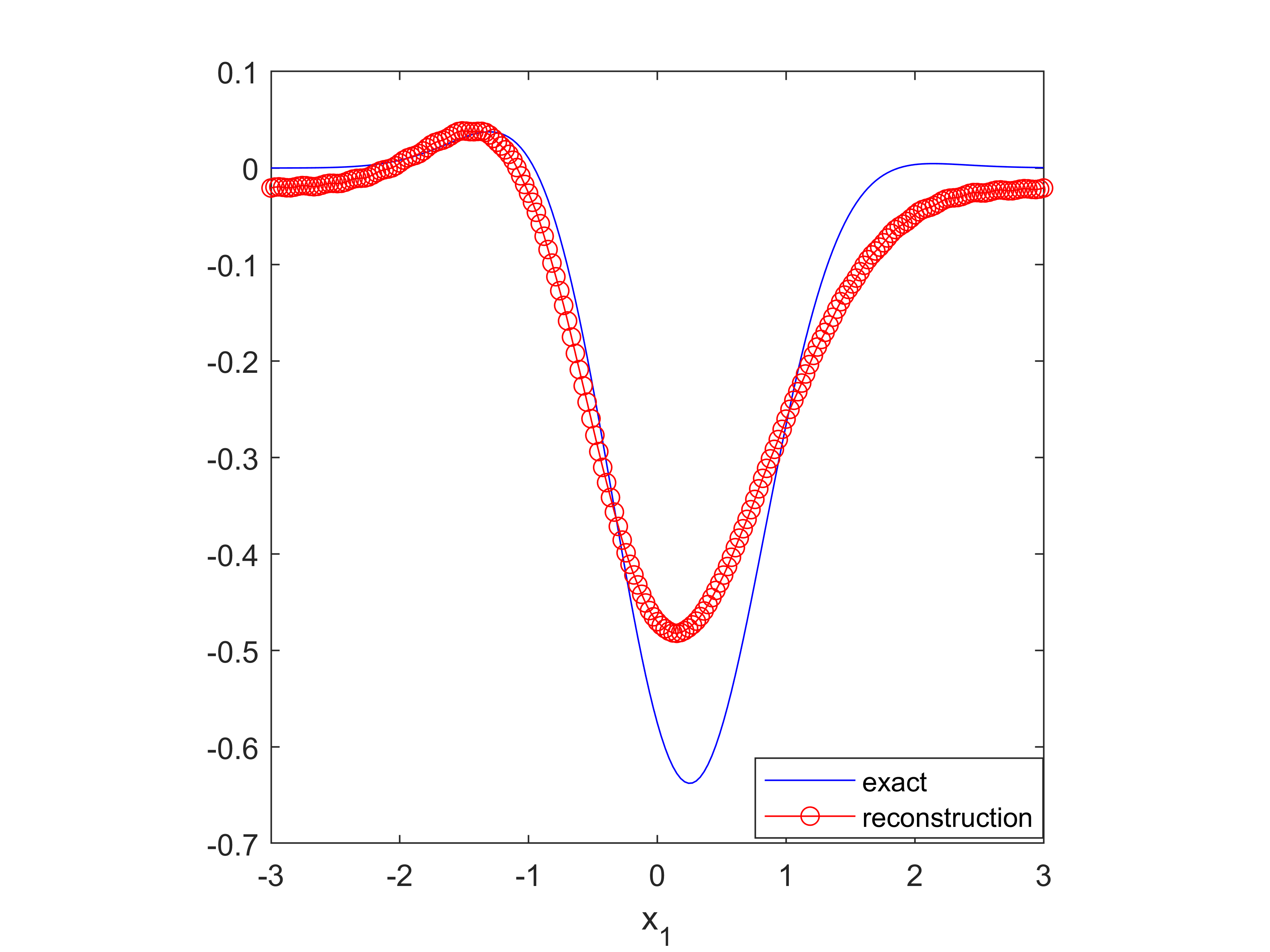}}
		\subfigure[]{\includegraphics[width=0.45\textwidth]{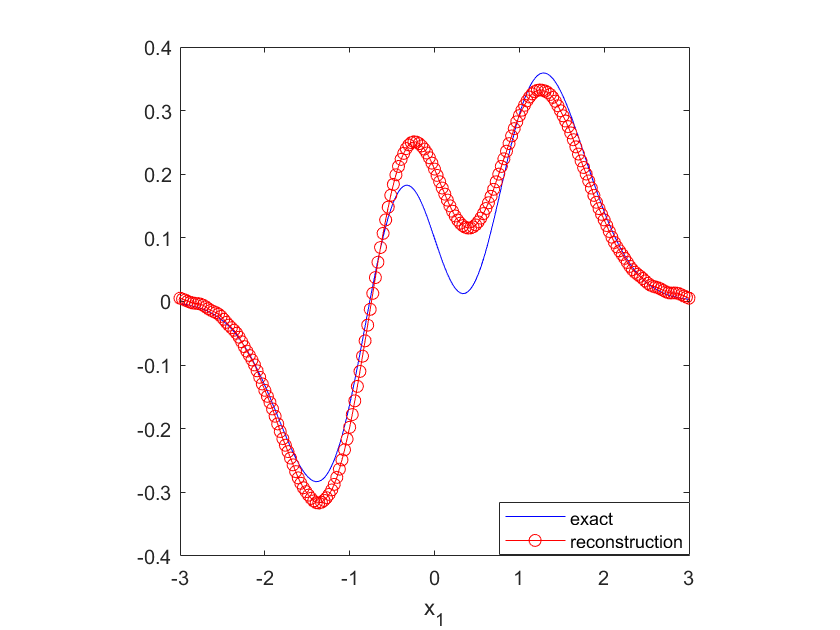}}
		\caption{The cross-section plots of source function $S_3$ and its reconstruction. (a)--(b): reconstruction under full-aperture observation; (c)--(d): reconstruction with $\theta_{\max}=\frac{3\pi}{2};$ (e)--(f): reconstruction with $\theta_{\max}=\pi$. The left column: cross-section plots at $x_2=0$; the right column: cross-section plots at $x_2=-1.5$.}\label{fig: S3_reconstruction2}
\end{figure}

\end{document}